\documentclass[10pt, a4paper]{article}
\usepackage{latexsym}
\usepackage{amsmath}
\usepackage{amssymb}
\usepackage[all]{xy}
\usepackage{bm}
\usepackage{amsfonts}
\usepackage{verbatim}
\usepackage{amsthm}
\usepackage{mathrsfs}
\usepackage{epsfig}
\usepackage{xy}
\usepackage{array}
\usepackage{stmaryrd}
\usepackage{graphicx,color}
\usepackage{xcolor}
\usepackage{tikz}
\usetikzlibrary{arrows,calc}
\usepackage{etex}
\usepackage{mathdots}
\usepackage{float}
\usepackage{graphics}
\usepackage{pdflscape}

\usepackage{anysize,hyperref}
\input xypic
\xyoption{all}

\usepackage[perpage,symbol]{footmisc}
\usepackage{setspace}
\usepackage{paralist}

\usepackage{indentfirst}

\def\X{\mathcal{X}}

\def\B{\mathscr{B}}
\def\C{\mathfrak{C}}
\def\E{\mathbb{E}}
\def\s{\mathfrak{s}}

\def\op{^\mathrm{op}}
\def\Ab{\mathit{Ab}}
\def\del{\delta}
\def\dr{\ar@{->}[r]}

\def\Y{\mathscr{Y}}

\def\X{\mathscr{X}}
\def\Y{\mathscr{Y}}

\def\Hom{\mbox{Hom}}

\newtheorem{theorem}{Theorem}[section]
\newtheorem{lemma}[theorem]{Lemma}
\newtheorem{corollary}[theorem]{Corollary}
\newtheorem{proposition}[theorem]{Proposition}

\theoremstyle{definition}
\newtheorem{definition}[theorem]{Definition}

\newtheorem{remark*}[]{Remark}
\newtheorem{example}[theorem]{Example}
\newtheorem{example*}[]{Example}

\newtheorem{construction*}[]{Construction}

\newtheorem{assumption*}[]{Assumption}

\begin{document}

\baselineskip=15pt
\title{\LARGE{\bf Tilting pairs in extriangulated categories}
\footnotetext{$^\ast$ Corresponding author.}
%\footnotetext{This work was supported by the NSF of China (Grant Nos. 11671221, 11971225, 11901341), the project ZR2019QA015 supported by Shandong Provincial Natural Science Foundation,
%and the Young Talents Invitation Program of Shandong Province. }
}
\medskip
\author{\normalsize{Tiwei Zhao$^\ast$, Bin Zhu, Xiao Zhuang}}

\date{}

\maketitle
\def\blue{\color{blue}}
\def\red{\color{red}}

\begin{abstract}
Extriangulated categories were introduced by Nakaoka and Palu to give a unification of properties in exact categories and extension-closed subcategories of triangulated categories. A notion of tilting pairs in an extriangulated category is introduced in this paper.  We give a Bazzoni  characterization  of tilting  pairs in this setting. We also obtain Auslander-Reiten correspondence   of tilting  pairs which  classifies finite $\mathcal{C}$-tilting subcategories for a certain self-orthogonal subcategory $\mathcal{C}$ with some assumptions. This generalizes  the known results given by Wei and  Xi for the categories of finitely generated modules over Artin algebras, thereby providing new insights in exact and triangulated categories.
\end{abstract}

\section{Introduction}

Exact categories and triangulated categories are two fundamental structures in algebra, geometry and topology.
 As expected, exact categories and triangulated categories are not
independent of each other.  In \cite{np}, Nakaoka and Palu introduced the notion of externally triangulated categories
(extriangulated categories for short) as a simultaneous generalization of exact categories and extension-closed subcategories of triangulated
categories (they may no longer be triangulated
categories in general). After that, the study of extriangulated categories has become an active topic, and   up to now,  many results on exact categories and triangulated
categories have gotten realization in the setting of extriangulated categories by many authors, e.g. see \cite{CZZ,inp,ln,np,NP20,PPPP,zh,zz} and other references.

Tilting modules or tilting functors (i.e. the functors induced by tilting modules as Hom-functors) as a
generalization of Bernstein-Gelfand-Ponomarev reflection functors \cite{bgp} were introduced by Auslander,
Platzeck and Reiten \cite{apr}, Brenner and Butler \cite{bb}, Happel and Ringel \cite{hr}. The classical tilting theory plays a crucial role in the representation theory
of algebras and related topics. When  studying tilting modules over Artin algebras, Miyashita introduced the notion of tilting
pairs in \cite{m2}. This notion is a generalization of tilting modules, and  it turn out to be useful for constructing tilting modules associated with a series of idempotent ideals in terms of tilting pairs. Among the interesting results available in tilting theory, there is a beautiful characterization
of tilting modules by  Colpi and Trlifaj \cite{ct} for the one dimensional case and
by Bazzoni in the general case \cite{ba}, which states that a module $T$ is  $n$-tilting if and only if its right orthogonal subcategory can be presented by a kind of exact sequences which are constructed from ${\rm add}T$. Another important result is the well-known Auslander-Reiten correspondence. It is originally
given by Auslander and Reiten \cite{ar}, which states that there is a one-one correspondence between  tilting (resp.
cotilting) modules and certain covariantly (resp. contravariantly) finite subcategories over the  category of finitely generated modules
of an Artin algebra. In \cite{wx,wx2}, Wei and Xi extended Bazzoni characterization and Auslander-Reiten correspondence of tilting modules to the setting of
Miyashita's tilting pairs. It also has several generalizaions in different settings, for example, Wei considered  semi-tilting complexes in the derived category of complexes over a ring $R$ \cite{w}, and Di et al considered silting subcategories in triangulated categories \cite{dlww}.

 In this paper, inspired by the philosophy of Wei and Xi in the categories of finitely generated modules of Artin algebras,  we expect to introduce
  the notion of tilting pairs in extriangulated categories, and give a framework on the Bazzoni characterization and Auslander-Reiten correspondence of tilting  pairs in this setting, which provides new insights in exact and triangulated categories.

 The paper is organized as follows. In Section 2, we summarize some basic definitions and propositions about extriangulated categories. In Section 3,
 we mainly study several kinds of subcategories relative to a self-orthogonal subcategory $\omega$. The first one is the subcategory $\widehat{\omega}$ such that each
 object in it admits a finite $\omega$-resolution. The second one is the subcategory ${_\omega\X}$  such that each
 object in it admits a  proper $\omega$-resolution. The third one is the right orthogonal subcategory $\omega^{\perp}$ of $\omega$.
 We also consider subcategories $\check{\omega}$, ${\X_\omega}$ and ${^\perp\omega}$ dually. When $\omega$ is self-orthogonal, these subcategories possess
 nice homological properties, such as the closure of extensions, direct summands,  cones of inflations or cocones of deflations (see Lemma \ref{ort} and Proposition \ref{cone-co}). As a key result, we show that $\widetilde{\omega}:=\check{\hat{\omega}}$ (resp. $\check{{_\omega\X}}$)
can be obtained by taking cones or cocones from $\widehat{\omega}$ (resp. ${_\omega\X}$) to $\check{\omega}$ (Proposition \ref{equ}).
In Section 4, we introduce the notion of tilting pairs in extriangulated categories. It unifies tilting pairs in the categories of finitely generated modules of Artin algebras \cite{m2}, silting complexes in the bounded homotopy category of finite generated projective $R$-modules over an associative ring \cite{AI}, and meanwhile $n$-tilting subcategories in extriangulated categories \cite{ZZhuang}. We give a Bazzoni characterization for $n$-tilting pairs in this setting (see Theorem \ref{6}).
In Section 5, we mainly study the Auslander-Reiten correspondence for tilting pairs, which classifies finite $\mathcal{C}$-tilting subcategories for a certain self-orthogonal subcategory $\mathcal{C}$ with some assumptions. More precisely,
 let $\mathcal{C}$ be self-orthogonal such that ${_\mathcal{C}\X}$ has a relative injective cogenerator and
$\check{{_\mathcal{C}\X}}$ is finite filtered. Then there are bijective correspondences as follows:
$$
\xymatrix@C=0.8cm@R=1cm{
   \left\{\mbox{finite  }\mathcal{C}\mbox{-tilting subcategories}\right\} \ar@{<=>}[r]^-{{\rm Thm.} \ref{4.12}}\ar@{<=>}[rd]_-{{\rm Thm.} \ref{fal}}& {\left\{ \begin{array}{c}
                                                                                             \mbox{subcategories }\mathcal{D}\mbox{ which are relative} \\
                                                                                             \mbox{coresolving, strongly covariantly}\\
                                                                                             \mbox{finite in }{_\mathcal{C}\X},\mbox{ and satisfies }\check{\mathcal{D}}=\check{{_\mathcal{C}\X}}
                                                                                           \end{array}\right\}}\ar@{<=>}[d]^-{{\rm Prop.} \ref{4.15}}\\
                                                                                          & {\left\{ \begin{array}{c}
                                                                                             \mbox{subcategories }\mathcal{D}\mbox{ which are relative} \\
                                                                                             \mbox{resolving, strongly contravariantly}\\ \mbox{finite in }{_\mathcal{C}\X},\mbox{ and are contained in }\hat{\mathcal{C}}
                                                                                           \end{array}\right\}.}
   }
 $$

\section{Preliminaries}

%Throughout the article, $\C$ denotes an additive category. All subcategories considered are full additive subcategories closed under isomorphisms.
%We denote by $\C(A,B)$ or $\Hom_{\C}(A,B)$ the set of morphisms from $A$ to $ B$ in $\C$.

We first recall some definitions and some basic properties of extriangulated categories from \cite{np}.

Let $\C$ be an additive category. Suppose that $\C$ is equipped with a biadditive functor $\E\colon\C\op\times\C\to\Ab$, where $\Ab$ is the category of abelian groups. For any pair of objects $A,C\in\C$, an element $\delta\in\E(C,A)$ is called an $\E$-extension. Zero element $0\in\E(C,A)$ is called the split $\E$-extension.
Since $\E$ is a bifunctor, for any $a\in\C(A,A')$ and $c\in\C(C',C)$, we have $\E$-extensions
$$ \E(C,a)(\del)\in\E(C,A')\ \ \text{and}\ \ \ \E(c,A)(\del)\in\E(C',A). $$
We abbreviate them to $a_\ast\del$ and $c^\ast\del$ respectively.
%For any $A,C\in\C$, the zero element $0\in\E(C,A)$ is called the spilt $\E$-extension.
For any $\delta\in \E(C,A), \delta'\in \E(C',A'),$ since $\C$ and $\E$ are additive, we can define the $\E$-extension
$$\delta\oplus\delta'\in\E(C\oplus C',A\oplus A').$$

\begin{definition}{\rm (\cite[Definition 2.3]{np})}\label{mo}
A \emph{morphism} $(a,c):\delta\to\delta'$ of $\E$-extensions $\delta\in \E(C,A), \delta'\in \E(C',A')$ is a pair of morphisms $a\in\C(A,A')$ and $c\in\C(C,C')$ in $\C$ satisfying $ a_\ast\del=c^\ast\del'. $
\end{definition}

Let $A,C\in\C$ be any pair of objects. Sequences of morphisms in $\C$
$$\xymatrix@C=0.7cm{A\ar[r]^{x} & B \ar[r]^{y} & C}\ \ \text{and}\ \ \ \xymatrix@C=0.7cm{A\ar[r]^{x'} & B' \ar[r]^{y'} & C}$$
are said to be \emph{equivalent} if there exists an isomorphism $b\in\C(B,B')$ which makes the following diagram commutative.
$$\xymatrix@C=20pt@R=20pt{
A \ar[r]^x \ar@{=}[d] & B\ar[r]^y \ar[d]_{\simeq}^{b} & C\ar@{=}[d]&\\
A\ar[r]^{x'} & B' \ar[r]^{y'} & C &}$$

We denote the equivalence class of $\xymatrix@C=0.7cm{A\ar[r]^{x} & B \ar[r]^{y} & C}$ by $[\xymatrix@C=0.7cm{A\ar[r]^{x} & B \ar[r]^{y} & C}]$.

\begin{definition}{\rm (\cite[Definition 2.9]{np})}\label{re}
For any $\E$-extension $\delta\in\E(C,A)$, one can associate an equivalence class $\s(\delta)=[\xymatrix@C=0.7cm{A\ar[r]^{x} & B \ar[r]^{y} & C}].$  This $\s$ is called a \emph{realization} of $\E$, if it satisfies the following condition:
\begin{itemize}
\item Let $\del\in\E(C,A)$ and $\del'\in\E(C',A')$ be any pair of $\E$-extensions with $$\s(\del)=[\xymatrix@C=0.7cm{A\ar[r]^{x} & B \ar[r]^{y} & C}],\ \ \ \s(\del')=[\xymatrix@C=0.7cm{A'\ar[r]^{x'} & B'\ar[r]^{y'} & C'}].$$
Then, for any morphism $(a,c)\colon\del\to\del'$, there exists $b\in\C(B,B')$ which makes the following diagram commutative:
\begin{equation}\label{reali}
\begin{split}
\xymatrix@C=20pt@R=20pt{
A \ar[r]^x \ar[d]^a & B\ar[r]^y \ar[d]^{b} & C\ar[d]^c&\\
A'\ar[r]^{x'} & B' \ar[r]^{y'} & C'. &}
\end{split}
\end{equation}
\end{itemize}
In this case, we say that the sequence $\xymatrix@C=0.7cm{A\ar[r]^{x} & B \ar[r]^{y} & C}$ \emph{realizes} \ $\del$, whenever it satisfies $\s(\del)=[\xymatrix@C=0.7cm{A\ar[r]^{x} & B \ar[r]^{y} & C}]$.
In the above situation, we say that the triple $(a,b,c)$ \emph{realizes} $(a,c)$.
\end{definition}

\begin{definition}{\rm (\cite[Definition 2.10]{np})}\label{rea}
Let $\C, \E$ be as above. A realization $\s$ of $\E$ is said to be \emph{additive} if the following conditions are satisfied:
\begin{itemize}
\item[{\rm (i)}] For any $A,C\in\C$, the split $\E$-extension $0\in \E(C,A)$ satisfies $\s(0)=0$.
\item[{\rm (ii)}] For any pair of $\E$-extensions $\del\in\E(A,C)$ and $\del'\in\E(A',C')$, we have
                  $$\s(\del\oplus \del')=\s(\del)\oplus \s(\del').$$.
\end{itemize}
\end{definition}

\begin{definition}{\rm (\cite[Definition 2.12]{np})}\label{ext}   %{\red \Large I have deleted those expressions of \rm (ET3)$\op$ and \rm (ET4)$\op$ }
We call the triple $(\C,\E,\s)$ an \emph{extriangulated category} if the following conditions are satisfied:
\begin{itemize}
\item[{\rm (ET1)}] $\E\colon\C\op\times\C\to\Ab$ is a biadditive functor.
\item[{\rm (ET2)}] $\s$ is an additive realization of $\E$.
\item[{\rm (ET3)}] Let $\del\in\E(C,A)$ and $\del'\in\E(C',A')$ be any pair of $\E$-extensions, realized as
$$ \s(\del)=[\xymatrix@C=0.7cm{A\ar[r]^{x} & B \ar[r]^{y} & C}],\ \ \s(\del')=[\xymatrix@C=0.7cm{A'\ar[r]^{x'} & B' \ar[r]^{y'} & C'}]. $$
For any commutative square
$$\xymatrix@C=20pt@R=20pt{
A \ar[r]^x \ar[d]^a & B\ar[r]^y \ar[d]^{b} & C&\\
A'\ar[r]^{x'} & B' \ar[r]^{y'} & C' &}$$
in $\C$, there exists a morphism $(a,c)\colon\del\to\del'$ which is realized by $(a,b,c)$.
\item[{\rm (ET3)$\op$}] Dual of \rm (ET3).
\item[{\rm (ET4)}] Let $\delta\in\E(A,D)$ and $\del'\in\E(B,F)$ be $\E$-extensions realized as
$$\s(\del)=[\xymatrix@C=0.7cm{A\ar[r]^{f} & B \ar[r]^{f'} & D}]\ \ \text{and}\ \ \ \s(\del')=[\xymatrix@C=0.7cm{B\ar[r]^{g} & C \ar[r]^{g'} & F}]$$
respectively. Then there exist an object $E\in\C$, a commutative diagram
$$\xymatrix@C=20pt@R=20pt{A\ar[r]^{f}\ar@{=}[d]&B\ar[r]^{f'}\ar[d]^{g}&D\ar[d]^{d}\\
A\ar[r]^{h}&C\ar[d]^{g'}\ar[r]^{h'}&E\ar[d]^{e}\\
&F\ar@{=}[r]&F}$$
in $\C$, and an $\E$-extension $\del^{''}\in\E(E,A)$ realized by $\xymatrix@C=0.7cm{A\ar[r]^{h} & C \ar[r]^{h'} & E},$ which satisfy the following compatibilities:
\begin{itemize}
\item[{\rm (i)}] $\xymatrix@C=0.7cm{D\ar[r]^{d} & E \ar[r]^{e} & F}$  realizes $f'_{\ast}\del'$,
\item[{\rm (ii)}] $d^\ast\del''=\del$,

\item[{\rm (iii)}] $f_{\ast}\del''=e^{\ast}\del'$.
\end{itemize}

\item[{\rm (ET4)$\op$}]  Dual of \rm (ET4).
\end{itemize}
\end{definition}

For an extriangulated category $\C$, we use the following notations.

\begin{itemize}
\item[(i)] A sequence $\xymatrix@C=0.41cm{A\ar[r]^{x} & B \ar[r]^{y} & C}$ is called a \emph{conflation} if it realizes some $\E$-extension $\del\in\E(C,A)$.
In which case, $\xymatrix@C=0.41cm{A\ar[r]^{x} & B}$ is called an \emph{inflation} and $\xymatrix@C=0.41cm{B \ar[r]^{y} & C}$ is called a \emph{deflation}. We call $\xymatrix@C=0.41cm{A\ar[r]^{x} & B \ar[r]^{y} & C\ar@{-->}[r]^{\del}&}$ an \emph{$\E$-triangle}.

\item[(ii)] Given an $\E$-triangle $\xymatrix@C=0.5cm{A\ar[r]^{x} & B \ar[r]^{y} & C\ar@{-->}[r]^{\del}&},$ we call $A$ the \emph{cocone} of $y\colon B\to C,$ and denote it by ${\rm Cocone}(y);$ we call $C$ the \emph{cone} of $x\colon A\to B,$ and denote it by ${\rm Cone}(x).$
\item[(iii)] A subcategory $\mathcal{T}$ of $\C$ is called \emph{extension-closed} if for any $\E$-triangle $\xymatrix@C=0.5cm{A\ar[r]^{x} & B \ar[r]^{y} & C\ar@{-->}[r]^{\del}&} $ with $A, C\in \mathcal{T}$, we have $B\in \mathcal{T}$.
\end{itemize}

\begin{lemma}\label{ile}{\rm (\cite[Corollary 3.15]{np})}
Let $(\C,\E,\s)$ be an extriangulated category. Then the following results hold.

\item{(1)}  Let $C$ be any object in $\C$, and let
$$\xymatrix@C=0.5cm{A_{1}\ar[r]^{x_{1}}&B_{1}\ar[r]^{y_{1}}&C\ar@{-->}[r]^{\delta_{1}}&},
\xymatrix@C=0.5cm{A_{2}\ar[r]^{x_{2}}&B_{2}\ar[r]^{y_{2}}&C\ar@{-->}[r]^{\delta_{2}}&}$$
be any pair of $\E$-triangles. Then there is a commutative diagram in $\C$
$$\xymatrix@C=20pt@R=20pt{&A_{2}\ar@{=}[r]\ar[d]^{m_{2}}&A_{2}\ar[d]^{x_{2}}\\
A_{1}\ar[r]^{m_{1}}\ar@{=}[d]&M\ar[r]^{e_{1}}\ar[d]^{e_{2}}&B_{2}\ar[d]^{y_{2}}&\\
A_{1}\ar[r]^{x_{1}}&B_{1}\ar[r]^{y_{1}}&C&\\
&&}$$ which satisfies
$$\s(y_{2}^{*}\delta_{1})=[\xymatrix@C=0.41cm{A_{1}\ar[r]^{m_{1}} & M \ar[r]^{e_{1}} & B_{2}}],$$
$$\s(y_{1}^{*}\delta_{2})=[\xymatrix@C=0.41cm{A_{2}\ar[r]^{m_{2}}&M\ar[r]^{e_{2}}&B_{1}}],$$
$$m_{1*}\delta_{1}+m_{2*}\delta_{2}=0.$$
\item{(2)}  Dual of {(1)}.
\end{lemma}

The following lemma was proved in \cite[Proposition 1.20]{ln}, see also \cite[Corollary 3.16]{np}.
\begin{lemma}\label{ilem}
Let $\xymatrix@C=0.5cm{A\ar[r]^{x}&B\ar[r]^{y}&C\ar@{-->}[r]^{\delta}&}$ be an $\E$-triangle,  $f\colon A\to D$  any morphism, and
$\xymatrix@C=0.5cm{D\ar[r]^{d}&E\ar[r]^{e}&C\ar@{-->}[r]^{f_{*}\delta}&}$  any $\E$-triangle realizing $f_{*}\delta.$ Then there is a morphism $g$
which gives a morphism of $\E$-triangles
$$\xymatrix@C=20pt@R=20pt{A\ar[r]^{x}\ar[d]^f&B\ar[r]^y\ar[d]^g&C\ar@{-->}[r]^\del\ar@{=}[d]&\\
D\ar[r]^{d}&E\ar[r]^{e}&C\ar@{-->}[r]^{f_{*}\del}&}
$$
and moreover, $\xymatrix{A\ar[r]^{\binom{-f}{x}\quad}&D\oplus B\ar[r]^{\quad(d,\ g)}&E\ar@{-->}[r]^{e^{*}\delta}&}$ becomes an $\E$-triangle.
\end{lemma}

\begin{definition}{\rm (\cite[Definition 3.23]{np})}\label{pro}
Let $\C, \E$ be as above. An object $P\in \C$ is called \emph{projective} if it satisfies the following condition.
\begin{itemize}
\item For any $\E$-triangle $\xymatrix@C=0.41cm{A\ar[r]^{x} & B \ar[r]^{y}&C\ar@{-->}[r]^{\delta}&}$ and any morphism $c\in \C(P,C)$, there exists $b\in \C(P,B)$ satisfying $y\circ b=c$.
\end{itemize}
The notion of \emph{injective} objects are defined dually.

We denote the subcategory consisting of projective (injective, resp.) objects in $\C$ by ${\rm Proj}(\C)$ (${\rm Inj}(\C)$, resp.).
\end{definition}

\begin{definition}\label{proje}  {\rm (\cite[Definition 3.25]{np})}
Let $(\C,\E,\s)$ be an extriangulated category. We say that it has \emph{enough projectives} (\emph{enough injectives}, resp.) if it satisfies the following condition:
\begin{itemize}
\item For any object $C\in \C$ ($A\in \C$, resp.), there exists an $\E$-triangle $$\xymatrix@C=0.5cm{A\ar[r]^{x}&P\ar[r]^{y}&C\ar@{-->}[r]^{\delta}&} (\xymatrix@C=0.5cm{A\ar[r]^{x}&I\ar[r]^{y}&C\ar@{-->}[r]^{\delta}&}, \mbox{resp.})$$ satisfying $P\in {\rm Proj}(\C)$ ($I\in {\rm Inj}(\C)$, resp.).

In this case, $A$ is called the \emph{syzygy} of $C$ ($C$ is called the \emph{cosyzygy} of $A$, resp.) and is denoted by $\Omega(C)$ $(\Sigma(A),$ resp.).
\end{itemize}
\end{definition}

Suppose $\C$ is an extriangulated category with enough projectives and injectives. For a subcategory $\B\subseteq\C$, put $\Omega^{0}\B=\B$, and for $i>0$ we define $\Omega^{i}\B$ inductively to be the subcategory consisting of syzygies of objects in $\Omega^{i-1}\B$, i.e.
$$\Omega^{i}\B=\Omega(\Omega^{i-1}\B).$$
We call $\Omega^{i}\B$ the $i$-th syzygy of $\B$. Dually we define the $i$-th cosyzygy $\Sigma^{i}\B$ by $\Sigma^{0}\B=\B$ and $\Sigma^{i}\B=\Sigma(\Sigma^{i-1}\B)$ for $i>0$.

In \cite[Proposition 5.2]{ln} the authors defined higher extension groups in an extriangulated category having enough projectives and injectives as $\E^{i+1}(X,Y)\cong\E(X,\Sigma^{i}Y)\cong \E(\Omega^{i}X,Y)$ for $i\geq0,$ and they showed the following result.

\begin{lemma}\label{long} {\rm (\cite[Proposition 5.2]{ln})}
Let $\xymatrix@C=0.5cm{A\ar[r]^{x}&B\ar[r]^{y}&C\ar@{-->}[r]^{\delta}&}$ be an $\E$-triangle. For any object $X\in \B$, there are long exact sequences
$$\cdots\rightarrow\E^{i}(X, A)\xrightarrow{x_{*}}\E^{i}(X, B)\xrightarrow{y_{*}}\E^{i}(X, C)\rightarrow\E^{i+1}(X, A)\xrightarrow{x_{*}}\E^{i+1}(X, B)\xrightarrow{y_{*}}\cdots (i\geq 0),$$
$$\cdots\rightarrow\E^{i}(C, X)\xrightarrow{y^{*}}\E^{i}(B, X)\xrightarrow{x^{*}}\E^{i}(A, X)\rightarrow\E^{i+1}(C, X)\xrightarrow{y^{*}}\E^{i+1}(B, X)\xrightarrow{x^{*}}\cdots (i\geq 0).$$
\end{lemma}

An \emph{$\E$-triangle sequence} in $\C$ is defined as a sequence
\begin{equation}\label{E}
  \cdots\rightarrow W_{n+1}\xrightarrow{d_{n+1}}W_{n}\xrightarrow{d_{n}}W_{n-1}\rightarrow\cdots
\end{equation}
over $\C$ such that for any $n$, there are $\E$-triangles $\xymatrix{K_{n+1}\ar[r]^{g_{n}}&W_{n}\ar[r]^{f_{n}}&K_{n}\ar@{-->}[r]^{\delta^{n}}&}$ and the differential $d_{n}=g_{n-1}f_{n}.$

\section{Basic results}
\setcounter{equation}{0}

Throughout this paper, we always assume that $\C$ is a Krull-Schmidt extriangulated category having enough projectives and injectives, and  a subcategory of $\C$ means a full and additive subcategory which is closed under isomorphisms and direct summands.

For any $T\in\C,$ we denote by $\mbox{add}(T)$ the category of objects isomorphic to direct summands of finite direct sums of $T.$

For a subcategory $\omega\subseteq\C,$ we define
\begin{align*}
   \omega^{\bot_{1}}=&\{Y\in\C\ | \ \E(X,Y)=0, \forall X\in\omega \},\\
  \omega^{\bot}=&\{Y\in\C \ | \ \E^{i}(X,Y)=0 \ , \forall i\geq1, \ \forall X\in\omega\}.
\end{align*}
Similarly, we define
\begin{align*}
{^{\bot_{1}}\omega}=&\{Y\ \in\C\ |\ \E(Y,X)=0, \forall X\in\omega \},\\
{^{\bot}\omega}=&\{Y\in\C\ |\ \E^{i}(Y,X)\ =0, \forall i\geq1, \ \forall X\in\omega\}.
\end{align*}

Let $\mathcal{X}$ and $\mathcal{Y}$ be two subcategories of $\C$.  If $\mathcal{Y}\subseteq \mathcal{X}^{\bot}$, or equivalently, $\mathcal{X}\subseteq {^{\bot}\mathcal{Y}}$, then we use the symbol $\mathcal{X}{\bot}\mathcal{Y}$.

Define
\begin{align*}
  \hat{{\omega}}_{n}=&\left\{A\in\C\mid\mbox{there is an $\E$-triangle sequence } W_{n}\xrightarrow{d_{n}}\cdots  \rightarrow W_{1}\xrightarrow{d_{1}}W_{0}\xrightarrow{d_{0}}A\right.\\
 &\ \ \ \ \ \ \ \ \ \ \ \left. \mbox{ with }W_{i}\in\omega \right\},\\
  \check{\omega}_{n}=&\left\{A\in\C\mid\mbox{there is an $\E$-triangle sequence }A\xrightarrow{d^{0}} W^{0}\xrightarrow{d^{1}}W^{1}\rightarrow\cdots \xrightarrow{d^{n}}W^{n}\right.\\
 &\ \ \ \ \ \ \ \ \ \ \ \left.\mbox{ with }W^{i}\in\omega\right\}.
\end{align*}

We denote by $\hat{\omega}$( $\check{\omega}, \mbox{resp.})$ the union of all $\hat{\omega}_{n}$ $(\check{\omega}_{n}, \mbox{resp.})$ for integers $n\geq 0$.  That is to say
$$\hat{\omega}=\bigcup_{n=0}^{\infty}\hat{\omega}_{n},\ \check{\omega}=\bigcup_{n=0}^{\infty}\check{\omega}_{n}.$$

We also denote by $\widetilde{\omega}=\check{\hat{\omega}}$.

%\begin{definition}
A subcategory $\omega$ is said to be \emph{self-orthogonal} provided that $\omega{\bot}\omega$.
%\end{definition}
%
A subcategory $\omega$ is said to be \emph{finite} if  $\omega=\mbox{add}(T)$ for some $T\in\C.$

Given a self-orthogonal subcategory $\omega.$ We define
\begin{align*}
  {_\omega\X}=&\left\{A\in\C\mid\mbox{there is an $\E$-triangle sequence }\cdots\rightarrow W_{2}\xrightarrow{d_{2}} W_{1}\xrightarrow{d_{1}}W_{0}\xrightarrow{d_{0}}A\right.\\
  &\left.\mbox{ as in (\ref{E}) with each } W_{i}\in\omega\mbox{ and each } K_i\in\omega^{\bot} \right\},\\
  {\X_\omega}=&\left\{A\in\C\mid\mbox{there is an $\E$-triangle sequence }A\xrightarrow{d^{0}} W^{0}\xrightarrow{d^{1}}W^{1}\xrightarrow{d^{2}} W^{2}\rightarrow\cdots\right.\\
  &\left.\mbox{ as in (\ref{E}) with each }W^{i}\in\omega\mbox{ and each } K^i\in{^{\bot}\omega} \right\}.
\end{align*}

It is obvious that $\omega\subseteq {_\omega\X}$ ($\omega\subseteq \X_{\omega}$) and ${_\omega\X}$ ($\X_{\omega}$, resp.) is the largest subcategory of $\C$ such that $\omega$ is projective (injective, resp.) and a generator (cogenerator, resp.) in it.
Moreover, we have
 $$\hat{\omega}\subseteq {_\omega\X}\subseteq\omega^{\bot}, \ \check{\omega}\subseteq\X_{\omega}\subseteq\sideset{^{\bot}}{}{\mathop{\omega}}.$$

We first collect some basic properties for these subcategories.

\begin{lemma}\label{ort}
Suppose that $\omega$ is a self-orthogonal subcategory of $\C.$ Then
\begin{itemize}
  \item[$(1)$] $\check{\omega}\bot \omega^{\bot}$. %$\E^{i}(X_{1},X_{2})=0$ for any $X_{1}\in\check{\omega}$, $X_{2}\in\omega^{\bot}$ and any $i\geq1$.
  \item[$(1')$] ${^{\bot}\omega}\bot \hat{\omega}$.   % $\E^{i}(X_{1},X_{2})=0$ for any $X_{1}\in{^{\bot}\omega}$, $X_{2}\in\hat{\omega}$ and any $i\geq1$.
  \item [$(2)$] $\check{\omega}_n$ ($\hat{\omega}_n$, resp.) is closed under extensions and direct summands.
  \item [$(3)$] ${}_\omega\X$ is closed under extensions, direct summands and cones of inflations.
  \item [$(3')$] $\X_\omega$ is closed under extensions, direct summands and cocones of deflations.
  \item [$(4)$] Given an $\E$-triangle $\xymatrix@C=0.5cm{X\ar[r]&Y\ar[r]&Z\ar@{-->}[r]&}$, if $Y,Z\in{_\omega\X}$ ($\hat{\omega}_n$, resp.)
  and $X\in\omega^{\bot_1}$, then $X\in{}_\omega\X$ ($\hat{\omega}_n$, resp.).
  \item [$(4')$]Given an $\E$-triangle $\xymatrix@C=0.5cm{X\ar[r]&Y\ar[r]&Z\ar@{-->}[r]&}$, if $X,Y\in{\X_\omega}$ ($\check{\omega}_n$, resp.)
  and $Z\in{^{\bot_1}\omega}$, then $Z\in{}\X_\omega$ ($\check{\omega}_n$, resp.).
\end{itemize}
\end{lemma}

\begin{proof}
(1) follows from \cite[Lemma 3.4]{ZZhuang}.

(2) We use induction on $n$. In case $n=0$, let $\xymatrix@C=0.5cm{W_0\ar[r]&W\ar[r]&W_1\ar@{-->}[r]&}$ be an $\E$-triangle with $W_0,W_1\in\omega$. Since
$\omega$ is self-orthogonal, it is split, and hence $W\cong W_0\oplus W_1\in\omega$. Assume the result holds for $n$.
Let $\xymatrix@C=0.5cm{X\ar[r]&Y\ar[r]&Z\ar@{-->}[r]&}$ be an $\E$-triangle with $X,Z\in\check{\omega}_{n+1}$. Then there are two $\E$-triangles
 $\xymatrix@C=0.5cm{X\ar[r]^{f_1}&W_0 \ar[r]^{g_1}&V_0\ar@{-->}[r]&}$ and $\xymatrix@C=0.5cm{Z\ar[r]^{f_2}&W_1\ar[r]^{g_2}&V_1\ar@{-->}[r]&}$ with $W_0,W_1\in\omega$ and $V_0,V_1\in\check{\omega}_{n}$.
 Consider the following commutative diagram
 \[\xymatrix@C=20pt@R=20pt{X\ar[r]\ar[d]&Y\ar[r]\ar[d]&Z\ar@{-->}[r]\ar@{=}[d]&\\
 W_0\ar[r]\ar[d]&W\ar[r]\ar[d]&Z\ar@{-->}[r]&\\
 V_0\ar@{=}[r]\ar@{-->}[d]&V_0\ar@{-->}[d]&&\\
 &&&
 }\]
 Since $Z\in\check{\omega}_{n+1}\subseteq{^\bot}\omega$ and $W_0\in\omega$, the middle row is split, and hence $W\cong W_0\oplus Z$.
 Consider the following commutative diagram
  \[\xymatrix@C=20pt@R=20pt{Y\ar@{=}[r]\ar[d]&Y\ar[d]&&\\
  W_0\oplus Z\ar[r]^{\tiny{
       1 \ \ 0  \choose 0  \ f_2}
  }\ar[d]&W_0\oplus W_1\ar[r]^{\tiny\ \ (0  \ g_2)}\ar[d]&V_1\ar@{-->}[r]\ar@{=}[d]&\\
 V_0\ar[r]\ar@{-->}[d]&V\ar[r]\ar@{-->}[d]&V_1\ar@{-->}[r]&\\
 &&&
 }\]
Since $V_0,V_1\in\check{\omega}_{n}$, $V\in\check{\omega}_{n}$ by the induction hypothesis. Moreover, $W_0\oplus W_1\in\omega$. Thus $Y\in\check{\omega}_{n+1}$.

Now we show that $\check{\omega}_n$ is closed under direct summands.
In fact we will show that
\begin{align*}
\check{\omega}_n&=\{X\in \X_{\omega}|\E^{n+1}(Y,X)=0, \mbox{for all}\ Y\in\sideset{^{\bot}}{}{\mathop{\omega}}\}\\
&=\{X\in \X_{\omega}|\E^{n+1}(Y,X)=0, \mbox{for all}\ Y\in {_\omega\X}\}.
\end{align*}
If $X\in\check{\omega}_n,$ we have an $\E$-triangle sequence
$$X\xrightarrow{d^{0}} W^{0}\xrightarrow{d^{1}}W^{1}\rightarrow\cdots  \rightarrow W^{n-1}\xrightarrow{d^{n}}W_{n}\ \ $$
with each $W^{i}\in\omega.$
Thus using $\C(Y,-)$ to this sequence for any $Y\in\sideset{^{\bot}}{}{\mathop{\omega}}$ {or} $Y\in {_{\omega}\X},$ we have $\E^{n+1}(Y,X)=\E(Y,W^{n})=0.$
On the other hand, for any $X\in \X_{\omega},$ there is an $\E$-triangle sequence
$$X\xrightarrow{d^{0}} W^{0}\xrightarrow{d^{1}}W^{1}\rightarrow\cdots  \rightarrow W^{n-1}\xrightarrow{d^{n}}W^{n}\rightarrow\cdots \ \ \ $$
with $W^{i}\in\sideset{^{\bot}}{}{\mathop{\omega}}.$
Then one can see that every corresponding $K^i$ as in (\ref{E}) belongs to $\X_{\omega}.$
If $\E^{n+1}(Y,X)=0$ for any $Y\in\sideset{^{\bot}}{}{\mathop{\omega}}$ {or} $Y\in {_{\omega}\X},$ then by applying $\C(K^{n+1},-)$ to this sequence, we have $\E(K^{n+1},K^{n})=\E^{n+1}(K^{n+1},X)=0.$
Therefore one can see that the $\E$-triangle $\xymatrix@C=0.5cm{K^{n}\ar[r]&W^{n}\ar[r]&K^{n+1}\ar@{-->}[r]&}$ splits.
Hence $K^{n}\in\omega$ and $X\in\check{\omega}_n.$
Using the above result, we have that $\check{\omega}_n$ is closed under direct summands.
Similarly, one can see that $\hat{\omega}_n$ is closed under extensions and direct summands.

(3) follows from \cite[Lemma 3.6]{ZZhuang}.

(4) Since $Z\in{_\omega\X}$, there is an $\E$-triangle $\xymatrix@C=0.5cm{K_1\ar[r]&W_0\ar[r]&Z\ar@{-->}[r]&}$ with $W_0\in\omega$ and $K_1\in{_\omega\X}$.
Consider the following commutative diagram
  \[\xymatrix@C=20pt@R=20pt{&K_1\ar@{=}[r]\ar[d]&K_1\ar[d]&\\
  X\ar[r]\ar@{=}[d]&W\ar[r]\ar[d]&W_0\ar@{-->}[r]\ar[d]&\\
 X\ar[r]&Y\ar[r]\ar@{-->}[d]&Z\ar@{-->}[r]\ar@{-->}[d]&\\
 &&&
 }\]
 Since $K_1,Y\in{_\omega\X}$, $W\in{_\omega\X}$ by (3). Moreover, $\E(W_0,X)=0$ implies $W\cong X\oplus W_0$, and hence $X\in{_\omega\X}$ by (3).
\end{proof}

\begin{lemma}\label{ch}
Let $\omega$ be self-orthogonal, $n$ a positive integer, $\Y$  a subcategory of $\C$ such that for any $Y\in\Y$, there is an $\E$-triangle $\xymatrix@C=0.5cm{Y'\ar[r]&W\ar[r]&Y\ar@{-->}[r]&}$ with $W\in\omega$ and $Y'\in\Y$. If there is an $\E$-triangle sequence
$$X\rightarrow N_{m}\rightarrow N_{m-1}\cdots\rightarrow N_{1}\rightarrow Z$$
for some positive integer $m$ with each $N_{i}\in\hat{\omega}_{n} (N_{i}\in{_\omega\X} \ \mbox{or} \ N_{i}\in\Y, \mbox{ resp}.)$,  then
\begin{itemize}
  \item[$(1)$] there exists an $\E$-triangle $\xymatrix@C=0.5cm{U\ar[r]&V\ar[r]&X\ar@{-->}[r]&}$ for some $U\in\hat{\omega}_{n-1}$  $(U\in{_\omega\X} \mbox{or}\ U\in\Y,\mbox{ resp}.)$ and for some $V$ such that there is an $\E$-triangle
  sequence
$$V\rightarrow M_{m}\rightarrow M_{m-1}\rightarrow\cdots\rightarrow M_{1}\rightarrow Z$$
 with each $M_{i}\in\omega.$

 If, moreover, $Z\in\hat{\omega}_{n+1}$ $( Z\in{_\omega\X} \mbox{or}\ Z\in\Y, \ \mbox{resp}.)$, then there exists an $\E$-triangle $\xymatrix@C=0.5cm{U\ar[r]&V\ar[r]&X\ar@{-->}[r]&}$ for some $U\in\hat{\omega}_{n-1}$ $(U\in{_\omega\X} \mbox{or}\ U\in\Y, \ \mbox{resp}.)$ and $V\in\check{\omega}_{m}$.
  \item[$(2)$] there exists an $\E$-triangle $\xymatrix@C=0.5cm{X\ar[r]&U\ar[r]&V\ar@{-->}[r]&}$ for some $U\in\hat{\omega}_{n}$ $(U\in{_\omega\X} \mbox{or}\ U\in\Y, \ \mbox{resp}.)$ and for some $V$ such that there is an $\E$-triangle
  sequence
$$V\rightarrow M_{m-1}\rightarrow M_{m-2}\rightarrow\cdots\rightarrow M_{1}\rightarrow Z$$
 with each $M_{i}\in\omega.$

 If, moreover, $Z\in\hat{\omega}_{n+1}$  $(Z\in{_\omega\X} \mbox{or}\ Z\in\Y, \ \mbox{resp}.)$, then there exists an $\E$-triangle $\xymatrix@C=0.5cm{X\ar[r]&U\ar[r]&V\ar@{-->}[r]&}$ for some $U\in\hat{\omega}_{n}$ $(U\in{_\omega\X} \mbox{or}\ U\in\Y, \ \mbox{resp}.)$ and $V\in\check{\omega}_{m-1}$.
\end{itemize}
\end{lemma}

\begin{proof} We only show the case of $\hat{\omega}.$
The case of ${_\omega\X}$ or $\Y$ can be proved similarly.

(1) We proceed by induction on $m$.

In case $m=1,$ there is an $\E$-triangle $\xymatrix@C=0.5cm{X\ar[r]&N_{1}\ar[r]&Z\ar@{-->}[r]&}$  with $N_{1}\in\hat{\omega}_{n}.$
Let $\xymatrix@C=0.5cm{U\ar[r]&M_{1}\ar[r]&N_{1}\ar@{-->}[r]&}$ be an $\E$-triangle with $M_{1}\in \omega$ and $U\in\hat{\omega}_{n-1}.$
Consider the following commutative diagram
\[\xymatrix@C=20pt@R=20pt{U\ar@{=}[r]\ar[d]&U\ar[d]&&\\
  V\ar[r]\ar[d]&M_1\ar[r]\ar[d]&Z\ar@{-->}[r]\ar@{=}[d]&\\
 X\ar[r]\ar@{-->}[d]&N_1\ar[r]\ar@{-->}[d]&Z\ar@{-->}[r]&\\
 &&&
 }\]
Then the first column and the second row are the desired $\E$-triangles.

Now suppose that the result holds for $m-1.$ Let $X'={\rm Cone}(X\to N_{m}),$ then there is an $\E$-triangle sequence
$$X'\rightarrow N_{m-1}\rightarrow\cdots\rightarrow N_{1}\rightarrow Z$$
 with each $N_{i}\in\hat{\omega}_{n}.$ By the inductive hypothesis, there are $\E$-triangles:
$\xymatrix@C=0.5cm{U'\ar[r]&V'\ar[r]&X'\ar@{-->}[r]&}$
and an $\E$-triangle sequence
$$V'\rightarrow M_{m-1}\rightarrow \cdots\rightarrow M_{1}\rightarrow Z$$
 with $U'\in\hat{\omega}_{n-1}$ and each $M_{i}\in\omega.$
Then  we have the following commutative diagram
$$\xymatrix@C=20pt@R=20pt{&U'\ar@{=}[r]\ar[d]&U'\ar[d]\\
X\ar[r]\ar@{=}[d]&Y\ar[r]\ar[d]&V'\ar[d]\ar@{-->}[r]&\\
X\ar[r]&N_{m}\ar[r]\ar@{-->}[d]&X'\ar@{-->}[r]\ar@{-->}[d]&&\\
&&&}$$
It is easy to check that   $Y\in\hat{\omega}_{n}$. % by Lemma \ref{ort}.
Thus
there exists an $\E$-triangle $\xymatrix@C=0.5cm{U\ar[r]&M_m\ar[r]&Y\ar@{-->}[r]&}$ with $M_m\in\omega$ and $U\in\hat{\omega}_{n-1}.$ Consider the commutative diagram
$$\xymatrix@C=20pt@R=20pt{U\ar@{=}[r]\ar[d]&U\ar[d]&\\
V\ar[r]\ar[d]&M_m\ar[d]\ar[r]&V'\ar@{=}[d]\ar@{-->}[r]&\\
X\ar[r]\ar@{-->}[d]&Y\ar[r]\ar@{-->}[d]&V'\ar@{-->}[r]&\\
&&&}$$
we can get an $\E$-triangle $\xymatrix@C=0.5cm{U\ar[r]&V\ar[r]&X\ar@{-->}[r]&}$ with $U\in\hat{\omega}_{n-1}$ and  $V$ satisfying that there is an $\E$-triangle
  sequence
$$V\rightarrow M_{m}\rightarrow M_{m-1}\rightarrow\cdots\rightarrow M_{1}\rightarrow Z$$
 with each $M_{i}\in\omega$.

Now assume  $Z\in\hat{\omega}_{n+1}$. Then there is an $\E$-triangle $\xymatrix@C=0.5cm{Z'\ar[r]&M_Z\ar[r]&Z\ar@{-->}[r]&}$ with $M_Z\in\omega$ and $Z'\in \hat{\omega}_{n}$.
Set $X_1={\rm CoCone}(N_1\to Z)$, and consider the following commutative diagram
\[\xymatrix@C=20pt@R=20pt{&Z'\ar@{=}[r]\ar[d]&Z'\ar[d]&\\
  X_1\ar[r]\ar@{=}[d]&Y'\ar[r]\ar[d]&M_Z\ar@{-->}[r]\ar[d]&\\
 X_1\ar[r]&N_1\ar[r]\ar@{-->}[d]&Z\ar@{-->}[r]\ar@{-->}[d]&\\
 &&&
 }\]
 Since $Z',N_1\in\hat{\omega}_{n}$, we have $Y'\in\hat{\omega}_{n}$. Moreover, we obtain an $\E$-triangle sequence
 $$X\rightarrow N_{m}\rightarrow \cdots\rightarrow N_2\rightarrow Y'\rightarrow M_Z.$$
 Thus, by the first statement, we can get an $\E$-triangle $\xymatrix@C=0.5cm{U\ar[r]&V\ar[r]&X\ar@{-->}[r]&}$ with $U\in\hat{\omega}_{n-1}$ and  $V$ satisfying that there is an $\E$-triangle
  sequences
$$V\rightarrow M_{m}\rightarrow M_{m-1}\rightarrow\cdots\rightarrow M_{1}\rightarrow M_Z$$
 with each $M_{i}\in\omega$. Moreover, $M_Z\in\omega$ implies $V\in\check{\omega}_{m}$.

 (2) By (1), there is  an $\E$-triangle $\xymatrix@C=0.5cm{U'\ar[r]&V'\ar[r]&X\ar@{-->}[r]&}$ with $U'\in\hat{\omega}_{n-1}$ and  $V'$ satisfying that there is an $\E$-triangle
  sequences
$$V'\rightarrow M_{m}\rightarrow M_{m-1}\rightarrow\cdots\rightarrow M_{1}\rightarrow Z$$
 with each $M_{i}\in\omega$. Set $V={\rm Cone}(V'\to M_m)$, and consider the following diagram
 $$\xymatrix@C=20pt@R=20pt{
U'\ar[r]\ar@{=}[d]&V'\ar[r]\ar[d]&X\ar[d]\ar@{-->}[r]&\\
U'\ar[r]&M_{m}\ar[r]\ar[d]&U\ar@{-->}[r]\ar[d]&\\
&V\ar@{=}[r]\ar@{-->}[d]&V'\ar@{-->}[d]&\\
&&&}$$
 Then $U'\in\hat{\omega}_{n-1}$ implies $U\in\hat{\omega}_{n}$.

 The proof of the second statement is similar to that of (1).
 \end{proof}

Now we give a relationship among $\widehat{\omega}$, $\check{\omega}$ and $\widetilde{\omega}$, which shows that $\widetilde{\omega}$
can be obtained by taking cones (or cocones) from $\widehat{\omega}$ to $\check{\omega}$.

\begin{proposition}\label{equ}
Let $\omega$ be self-orthogonal. The following statements are equivalent.
\begin{itemize}
  \item [$(1)$] $M\in\widetilde{\omega}$ (resp. $M\in \check{{_\omega\X}}$).
  \item [$(2)$] There is an $\E$-triangle $\xymatrix@C=0.5cm{U\ar[r]&V\ar[r]&M\ar@{-->}[r]&}$ with $U\in \widehat{\omega}$ (resp. $U\in{_\omega\X}$) and $V\in \check{\omega}$.
  \item [$(3)$] There is an $\E$-triangle $\xymatrix@C=0.5cm{M\ar[r]&U'\ar[r]&V'\ar@{-->}[r]&}$ with $U'\in \widehat{\omega}$ (resp. $U'\in{_\omega\X}$) and $V'\in \check{\omega}$.
\end{itemize}
\end{proposition}

\begin{proof}
(1) $\Rightarrow$ (2) Since $M\in\widetilde{\omega}=\check{\hat{\omega}}$, there exists integers $m,n$ such that there is an $\E$-triangle sequence
$$M\rightarrow N_{m}\rightarrow N_{m-1}\cdots\rightarrow N_{1}\rightarrow N_0$$
with each $N_i\in\hat{\omega}_n$. Thus, by Lemma \ref{ch}(1), there is an $\E$-triangle $\xymatrix@C=0.5cm{U\ar[r]&V\ar[r]&M\ar@{-->}[r]&}$
with $U\in \widehat{\omega}$ (resp. $U\in{_\omega\X}$) and $V\in \check{\omega}$.

(2) $\Rightarrow$ (3) It follows from the proof of Lemma \ref{ch}(2).

(3) $\Rightarrow$ (1) Since $V'\in \check{\omega}$, there is an an $\E$-triangle sequence
$$V'\rightarrow C^0\rightarrow C^1 \rightarrow\cdots\rightarrow C^m$$
with each $C^i\in\omega$. Composing it with the $\E$-triangle $\xymatrix@C=0.5cm{M\ar[r]&U'\ar[r]&V'\ar@{-->}[r]&}$, we can deduce that $M\in\widetilde{\omega}$.
\end{proof}

Using it, we can get the following equalities.

\begin{proposition}\label{3.4}
Let $\omega$ be self-orthogonal.
\begin{itemize}
  \item [$(1)$] ${^\bot({_\omega\X})}\cap \check{{_\omega\X}}=\check{\omega}$.
  \item [$(2)$] $({^\bot({_\omega\X})})^\bot\cap \check{{_\omega\X}}={_\omega{\X}}$.
\end{itemize}
\end{proposition}

\begin{proof}
(1) It is easy to check that $\check{\omega}\subseteq {^\bot({_\omega\X})}$. Moreover, since $\omega\subseteq {_\omega\X}$, we have $\check{\omega}\subseteq \check{{_\omega\X}}$. Thus
$\check{\omega}\subseteq {^\bot({_\omega\X})}\cap\check{{_\omega\X}}$. Conversely, let $M\in{^\bot({_\omega\X})}\cap\check{{_\omega\X}}$. By Proposition \ref{equ}, there is an $\E$-triangle $\xymatrix@C=0.5cm{U\ar[r]&V\ar[r]&M\ar@{-->}[r]&}$ with $U\in {_\omega\X}$  and $V\in \check{\omega}$. $M\in {^\bot({_\omega\X})}$ implies that $V\cong U\oplus M$, and hence $M\in\check{\omega}$ by Lemma \ref{ort}(2).
Thus ${^\bot({_\omega\X})}\cap\check{{_\omega\X}}\subseteq\check{\omega}$. Therefore, $\check{\omega}= {^\bot({_\omega\X})}\cap\check{{_\omega\X}}$.

(2) Clearly, ${_\omega{\X}}\subseteq({^\bot({_\omega\X})})^\bot\cap \check{{_\omega\X}}$. Conversely, let $M\in({^\bot({_\omega\X})})^\bot\cap \check{{_\omega\X}}$. By Proposition \ref{equ}, there is an $\E$-triangle $\xymatrix@C=0.5cm{M\ar[r]&U'\ar[r]&V'\ar@{-->}[r]&}$ with  $U'\in{_\omega\X}$ and $V'\in \check{\omega}$. Since $V'\in\check{\omega}\subseteq {^\bot({_\omega\X})}$, we have $\E(V',M)=0$ and hence $U'\cong M\oplus V'$. Thus
$M\in{_\omega{\X}}$, and so $({^\bot({_\omega\X})})^\bot\cap \check{{_\omega\X}}\subseteq{_\omega{\X}}$. Therefore, $({^\bot({_\omega\X})})^\bot\cap \check{{_\omega\X}}={_\omega{\X}}$.
\end{proof}

\begin{proposition}\label{cone-co}
Let $\omega$ be self-orthogonal. Then $\check{{_\omega\X}}$ (resp. $\widetilde{\omega}$) is closed under extensions, cocones of deflations and cones of inflations.
\end{proposition}

\begin{proof}
Let $\xymatrix@C=0.5cm{L\ar[r]&M\ar[r]&N\ar@{-->}[r]&}$ be an $\E$-triangle.

(1) If $L,N\in\check{{_\omega\X}}$, then by Proposition \ref{equ}, there exist $\E$-triangles $\xymatrix@C=0.5cm{L\ar[r]&X_L\ar[r]&L'\ar@{-->}[r]&}$, and $\xymatrix@C=0.5cm{X_N\ar[r]&N'\ar[r]&N\ar@{-->}[r]&}$
with $L',N'\in\check{\omega}$ and $X_L,X_N\in{_\omega{\X}}$. Consider the following commutative diagrams
\[\xymatrix@C=20pt@R=20pt{L\ar[r]\ar[d]&M\ar[r]\ar[d]&N\ar@{-->}[r]\ar@{=}[d]&\\
 X_L\ar[r]\ar[d]&Y\ar[r]\ar[d]&N\ar@{-->}[r]&\\
 L'\ar@{=}[r]\ar@{-->}[d]&L'\ar@{-->}[d]&&\\
 &&&
 }\]
 and
 \[\xymatrix@C=20pt@R=20pt{&X_N\ar@{=}[r]\ar[d]&X_N\ar[d]&\\
  X_L\ar[r]\ar@{=}[d]&Y'\ar[r]\ar[d]&N'\ar@{-->}[r]\ar[d]&\\
 X_L\ar[r]&Y\ar[r]\ar@{-->}[d]&N\ar@{-->}[r]\ar@{-->}[d]&\\
 &&&
 }\]
 Since $X_L\in{\check{\X}}\subseteq\omega^\bot$ and $N'\in\check{\omega}$, $\E(N',X_L)=0$ by Lemma \ref{ort}. Thus $Y'\cong X_L\oplus N'\in\check{{_\omega\X}}$ since $N'\in\check{\omega}\subseteq\check{{_\omega\X}}$.
 Then there is an $\E$-triangle $\xymatrix@C=0.5cm{Y'\ar[r]&X_{Y'}\ar[r]&Z\ar@{-->}[r]&}$ with $X_{Y'}\in{_\omega{\X}}$ and $Z\in\check{\omega}$ by Proposition \ref{equ}. Consider the following commutative diagrams
 $$\xymatrix@C=20pt@R=20pt{
X_N\ar[r]\ar@{=}[d]&Y'\ar[r]\ar[d]&Y\ar[d]\ar@{-->}[r]&\\
X_N\ar[r]&X_{Y'}\ar[r]\ar[d]&X_Z\ar@{-->}[r]\ar[d]&\\
&Z\ar@{=}[r]\ar@{-->}[d]&Z\ar@{-->}[d]&\\
&&&}$$
and
$$\xymatrix@C=20pt@R=20pt{
M\ar[r]\ar@{=}[d]&Y\ar[r]\ar[d]&L'\ar[d]\ar@{-->}[r]&\\
M\ar[r]&X_Z\ar[r]\ar[d]&M'\ar@{-->}[r]\ar[d]&\\
&Z\ar@{=}[r]\ar@{-->}[d]&Z\ar@{-->}[d]&\\
&&&}$$
Since $X_N,X_{Y'}\in{_\omega{\X}}$, $X_Z\in{_\omega{\X}}$ by Lemma \ref{ort}(3). Since $L',Z\in\check{\omega}$, $M'\in\check{\omega}$  by Lemma \ref{ort}(2). Thus by Proposition \ref{equ}, $M\in\check{{_\omega\X}}$.

(2) Assume $M,N\in\check{{_\omega\X}}$. By Proposition \ref{equ}, there is an $\E$-triangle $\xymatrix@C=0.5cm{M\ar[r]&U'\ar[r]&V'\ar@{-->}[r]&}$ with $U'\in{_\omega{\X}}$ and $V'\in\check{\omega}$.
Consider the following commutative diagram
$$\xymatrix@C=20pt@R=20pt{
L\ar[r]\ar@{=}[d]&M\ar[r]\ar[d]&N\ar[d]\ar@{-->}[r]&\\
L\ar[r]&U'\ar[r]\ar[d]&U_N\ar@{-->}[r]\ar[d]&\\
&V'\ar@{=}[r]\ar@{-->}[d]&V'\ar@{-->}[d]&\\
&&&}$$
Since $N\in\check{{_\omega\X}}$ and $V'\in\check{\omega}\subseteq\check{{_\omega\X}}$, $U_N\in\check{{_\omega\X}}$ by (1). Then  there is an $\E$-triangle $\xymatrix@C=0.5cm{U\ar[r]&V\ar[r]&U_{N}\ar@{-->}[r]&}$ with $U\in{_\omega{\X}}$ and $V\in\check{\omega}$ by Proposition \ref{equ}. Consider the following commutative diagram
\[\xymatrix@C=20pt@R=20pt{&U\ar@{=}[r]\ar[d]&U\ar[d]&\\
  L\ar[r]\ar@{=}[d]&W\ar[r]\ar[d]&V\ar@{-->}[r]\ar[d]&\\
 L\ar[r]&U'\ar[r]\ar@{-->}[d]&U_N\ar@{-->}[r]\ar@{-->}[d]&\\
 &&&
 }\]
 Since $U,U'\in{_\omega{\X}}$, $W\in{_\omega{\X}}$. By Proposition \ref{equ}, $L\in\check{{_\omega\X}}$.

 (3) Similar to (2).
\end{proof}

Let $M\in\check{\omega}$ (resp. $M\in\hat{\omega}$). We denote by
\begin{align*}
  \omega\mbox{-codim}M= & \min\left\{n\mid \mbox{there is an $\E$-triangle sequence }M\to W^0 \to W^1\to \cdots\to W^n \right.\\
  & \ \ \ \ \ \ \ \ \ \ \left.\mbox{ with each }W^i\in\omega\right\} \\
 \omega\mbox{-dim}M= &\min\left\{n\mid \mbox{there is an $\E$-triangle sequence }W_n\to \cdots\to W_1\to W_0 \to M \right.\\
 &\ \ \ \ \ \ \ \ \ \ \left.\mbox{ with each }W_i\in\omega\right\}.
\end{align*}
Similarly, if $\mathcal{C}\subseteq\check{\omega}$ (resp. $\mathcal{C}\subseteq\hat{\omega}$), we can define $\omega\mbox{-codim}\mathcal{C}$ (resp. $\omega\mbox{-dim}\mathcal{C}$) to be the smallest integer $n$ such that $\omega\mbox{-codim}M\leq n$ (resp. $\omega\mbox{-dim}M\leq n$) for all $M\in\mathcal{C}$.

The following lemma can be easily checked.

\begin{lemma}\label{3.6}
Let $\omega$ be a subcategory closed under extensions and cones of inflations. Assume that $\check{\omega}$ is closed under extensions, cones of inflations, and cocones of deflations.
Let  $\xymatrix@C=0.5cm{L\ar[r]&M\ar[r]&N\ar@{-->}[r]&}$ be an $\E$-triangle in $\check{\omega}$. Set $\omega\mbox{-}{\rm codim}L=l$, $\omega\mbox{-}{\rm codim}M=m$, $\omega\mbox{-}{\rm codim}N=n$.
\begin{itemize}
  \item [$(1)$] If $l=0$, then m=n.
  \item [$(2)$] If $m=0$ and $l>0$, then $l=n+1$.
  \item [$(3)$] If $n=0$ and $m>0$, then $l=m$.
  \item [$(4)$] If $l\leq \min\{m,n\}$, then $m=n$.
  \item [$(5)$] If $m<l$, then $l=n+1$.
  \item [$(6)$] If $n<m$, then $l=m$.
  \item [$(7)$] $m\leq\max\{l,n\}$.
  \item [$(8)$] $l\leq\max\{m,n\}+1$.
  \item [$(9)$] $n\leq\max\{l,m\}+1$.
\end{itemize}
\end{lemma}
\begin{proof}
We only prove (1), and the others are left to the reader. Assume $l=0$, i.e. $L\in\omega$.
It is easy to see that $m\geq n$.
We will use induction to prove the statement.
If $m=0$, then $n=0$ since $\omega$ is closed under cones of inflations; on the other hand, if $n=0$, then $m=0$ since $\omega$ is closed under extensions, i.e. we always have that $m=n$ if $m=0$ or $n=0$.
Suppose that the statement is true for any $i\leq m-1$.
Now we see the case $m$.
Let $\xymatrix@C=0.5cm{M\ar[r]&W_{0}\ar[r]&M_{0}\ar@{-->}[r]&}$ be an $\E$-triangle with $W_{0}\in\omega$ and $\omega\mbox{-}{\rm codim}M_{0}=m-1$, then we have the following commutative diagram:
$$\xymatrix@C=20pt@R=20pt{
L\ar[r]\ar@{=}[d]&M\ar[r]\ar[d]&N\ar[d]\ar@{-->}[r]&\\
L\ar[r]&W_{0}\ar[r]\ar[d]&E\ar@{-->}[r]\ar[d]&\\
&M_{0}\ar@{=}[r]\ar@{-->}[d]&M_{0}\ar@{-->}[d]&\\
&&&}$$
It is easy to see that $E\in\omega$ since $\omega$ is closed under cones of inflations.
We claim that $\omega\mbox{-}{\rm codim}N:=n=m$. Otherwise $n\leq m-1$ since $n\leq m$.
Then we have an $\E$-triangle $\xymatrix@C=0.5cm{N\ar[r]&W_{0}'\ar[r]&N_{0}'\ar@{-->}[r]&}$ with $W_{0}'\in\omega$ and $\omega\mbox{-}{\rm codim}N_{0}'\leq m-2$.
Using Lemma \ref{ile}(2), we have the following commutative diagram:
\[\xymatrix@C=20pt@R=20pt{N\ar[r]\ar[d]&E\ar[r]\ar[d]&M_{0}\ar@{-->}[r]\ar@{=}[d]&\\
 W_0'\ar[r]\ar[d]&F\ar[r]\ar[d]&M_{0}\ar@{-->}[r]&\\
 N_0'\ar@{=}[r]\ar@{-->}[d]&N_0'\ar@{-->}[d]&&\\
 &&&
 }\]
Using the induction hypothesis to the middle row, we have $\omega\mbox{-}{\rm codim}F=m-1$. Using the induction hypothesis to the middle column, we have
$\omega\mbox{-}{\rm codim}N_{0}'=\omega\mbox{-}{\rm codim}F=m-1$, which is a contradiction.
\end{proof}

\section{Tilting pairs of subcategories}

In this section, we begin with the definition of  tilting pairs of subcategories in an extriangulated category $\C$. Then we formulate the Bazzoni  characterization in this setting.

\subsection{$n$-tilting pairs}

Now we give the definition of tilting pairs of subcategories in an extriangulated category as follows.
\begin{definition}
A pair $(\mathcal{C},\mathcal{T})$ of subcategories is called a \emph{tilting pair} if
\begin{itemize}
  \item [$(1)$] $\mathcal{C}$ is self-orthogonal.
  \item [$(2)$] $\mathcal{T}$ is self-orthogonal.
  \item [$(3)$] $\mathcal{T}\subseteq \hat{\mathcal{C}}$.
  \item [$(4)$] $\mathcal{C}\subseteq \check{\mathcal{T}}$.
\end{itemize}

In this case, we say that $\mathcal{T}$ is a \emph{$\mathcal{C}$-tilting} subcategory, and $\mathcal{C}$ is a \emph{$\mathcal{T}$-cotilting} subcategory.
\end{definition}

If $(\mathcal{C},\mathcal{T})$ is a tilting pair such that $\mathcal{C}\mbox{-dim}\mathcal{T}\leq n$, then $(\mathcal{C},\mathcal{T})$ is said to be  an \emph{$n$-tilting pair}. Similarly, if $\mathcal{T}\mbox{-codim}\mathcal{C}\leq n$, we say that $(\mathcal{C},\mathcal{T})$ is an \emph{$n$-cotilting pair}.

\begin{example}\label{exampair}
\begin{itemize}
\item [(1)]Let $\C=\mbox{mod}A$, where $A$ is an Artin algebra and $\mbox{mod}A$ is the category of finitely generated left $A$-modules. Then tilting pairs defined in \cite{m2} and \cite{wx} are  examples of tilting pairs in an extriangulated category.
\item [(2)]Let $A$ be a finite dimensional $k$-algebra, and ${K^b({\rm proj}A)}$ be the bounded homotopy category of finite generated projective $A$-modules. Recall from \cite{BZ} that
a complex ${\bf P}=\{P^i\}$ is called a 2-term silting complex if
\begin{itemize}
  \item [(i)] $P^i=0$ for $i\neq -1,0$,
  \item [(ii)] $\Hom_{K^b({\rm proj}A)}(P,P[1])=0$,
  \item [(iii)] $\mbox{thick}{\bf P}={K^b({\rm proj}A)}$, where $\mbox{thick}{\bf P}$ is the smallest triangulated subcategory closed under direct summands containing ${\bf P}$.
\end{itemize}
By \cite[Theorem 1.1]{BZ},  any 2-term silting complex ${\bf P}$ is ${\rm proj}A$-tilting. In fact, for any 2-term silting complex ${\bf P}$, the pair
$({\rm add} A, {\rm add} {\bf P})$ is a $1$-tilting pair in ${K^b({\rm proj}A)}$.
\item [(3)] More generally, let $R$ be an associative ring with identity, and ${K^b({\rm proj}R)}$ be the bounded homotopy category of finite generated projective $R$-modules. If ${\bf T}$ is a silting complex (see \cite[Definition 2.1]{AI} for the definition of silting objects in triangulated categories) in ${K^b({\rm proj}R)}$ with $\inf\{i\in\mathbb{Z}\mid T_i\neq 0\}=l$ and $\sup\{i\in\mathbb{Z}\mid T_i\neq 0\}=k$, then the pair $({\rm add}R, {\rm add}{\bf T}[k])$ is a $(k-l)$-tilting pair by \cite[Corollary 4.8]{dlww}.
\item[(4)] Let $A$ be an Artin algebra, and $C,M\in\mbox{mod}A$. If $(C,M)$ is an $n$-tilting pair in $\mbox{mod}A$ in the sense of \cite{wx}, then $({\rm add}C, {\rm add}M)$ is an $n$-tilting pair in the bounded derived category $D^b(A)$.
\item [(5)] When $\C$ is an extriangulated category with enough projectives and injectives, Zhu and Zhuang \cite{ZZhuang} defined $n$-tilting subcategories $\mathcal T$ in $\C$. If in addition $\mathcal T=add(T)$, we call it an $n$-tilting object. Now we set $\mathcal{C}={\rm Proj}(\C)$ and $\mathcal T$ be an $n$-tilting subcategory in the sense of \cite[Definition 7]{ZZhuang}, then using Remark 4 in \cite{ZZhuang}, we have $(\mathcal{C},\mathcal T)$ is an $n$-tilting pair in $\C$.
\end{itemize}
\end{example}

In what follows,
we always assume that the following {\it weak idempotent completeness} (WIC for short)
given originally in \cite[Condition 5.8]{np} holds true on $\C$.

{\bf WIC Condition}: Let $f\in\C(A,B)$, $g\in\C(B,C)$.
\begin{itemize}
  \item  If $gf$ is an inflation, then so is $f$.
  \item  If $gf$ is a deflation, then so is $g$.
\end{itemize}

%In what follows, we always assume $\C$ satisfies Condition (WIC).

\begin{definition}{\rm (\cite[Definition 3.19]{zz})}
Let $\omega$ be a subcategory of $\C$. $\omega$ is called \emph{strongly contravariantly} (resp. \emph{strongly covariantly}) \emph{finite} if for any $C\in\C$, there is an $\E$-triangle $\xymatrix@C=0.5cm{K\ar[r]&W\ar[r]^g&C\ar@{-->}[r]&}$ (resp. $\xymatrix@C=0.5cm{C\ar[r]^g&W\ar[r]&V\ar@{-->}[r]&}$), where $g$ is a right (resp. left) $\omega$-approximation of $C$.
\end{definition}

From now on we assume that $\mathcal C$ is self-orthogonal.

\begin{lemma}\label{4.2}
Let $\mathcal{C}$ be self-orthogonal and $\omega\subseteq{_\mathcal{C}\X}$. If $\omega$ is a strongly contravariantly finite self-orthogonal subcategory, and $\mathcal{C}\subseteq\check{\omega}$, then $\omega^\bot\subseteq\mathcal{C}^\bot$ and $\omega^\bot\cap{_\mathcal{C}\X}={_\omega\X}$.

In particular, if $\mathcal{T}$ is a strongly contravariantly finite $\mathcal{C}$-tilting subcategory, then $\mathcal{T}^\bot\subseteq\mathcal{C}^\bot$ and $\mathcal{T}^\bot\cap{_\mathcal{C}\X}={_\mathcal{T}\X}$.
\end{lemma}

\begin{proof}
Since $\mathcal{C}\subseteq\check{\omega}$, let $M\in\omega^\bot$, then by Lemma \ref{ort}(1) $\E^i(C,M)=0$ for any $C\in\mathcal{C}$ and any $i\geq 1$. Thus $\omega^\bot\subseteq\mathcal{C}^\bot$.

$\omega^\bot\cap{_\mathcal{C}\X}\subseteq{_\omega\X}$:  Denote by
\[\mbox{Gen}\omega=\{M\mid\mbox{there is a deflation }W\to M\mbox{ with }W\in\omega\}.\]
Let $M\in\omega^\bot\cap{_\mathcal{C}\X}$. We claim $M\in\mbox{Gen}\omega$. In fact, choose an $\E$-triangle
$\xymatrix@C=0.5cm{M_1\ar[r]&C_0\ar[r]&M\ar@{-->}[r]&}$ with $C_0\in\mathcal{C}$ and $M_1\in{_\mathcal{C}\X}$. Since $C_0\in\check{\omega}$, there is an $\E$-triangle
$\xymatrix@C=0.5cm{C_0\ar[r]&T'\ar[r]&X\ar@{-->}[r]&}$ with $T'\in\omega$ and $X\in\check{\omega}$. Consider the following commutative diagram
$$\xymatrix@C=20pt@R=20pt{
M_1\ar[r]\ar@{=}[d]&C_0\ar[r]\ar[d]&M\ar[d]\ar@{-->}[r]&\\
M_1\ar[r]&T'\ar[r]\ar[d]&Y\ar@{-->}[r]\ar[d]&\\
&X\ar@{=}[r]\ar@{-->}[d]&X\ar@{-->}[d]&\\
&&&}$$
By Lemma \ref{ort}(1) $\E(X,M)=0$, and hence $M\oplus X\cong Y\in\mbox{Gen}\omega$.
Consider the following commutative diagram
$$\xymatrix@C=20pt@R=20pt{
M_1\ar[r]\ar@{=}[d]&X'\ar[r]\ar[d]&X\ar[d]^{1\choose 0}\ar@{-->}[r]&\\
M_1\ar[r]&T'\ar[r]\ar[d]&Y\ar@{-->}[r]\ar[d]^{(0 \ 1)}&\\
&M\ar@{=}[r]\ar@{-->}[d]&M\ar@{-->}[d]&\\
&&&}$$
Then we know that $M\in\mbox{Gen}\omega$. Consequently, we may take an $\E$-triangle
$\xymatrix@C=0.5cm{M'\ar[r]&T_M\ar[r]&M\ar@{-->}[r]&}$ with $M'\in\omega^\bot$ and $T_M\in\omega$ since $\omega$ is strongly contarvariantly finite. As $\omega^\bot\subseteq\mathcal{C}^\bot$ and $\omega\subseteq{_\mathcal{C}\X}$, we get $M'\in\omega^\bot\cap{_\mathcal{C}\X}$ by Lemma \ref{ort}. Now repeating the process to $M'$, and so on, we get $M\in{_\omega\X}$.

${_\omega\X}\subseteq\omega^\bot\cap{_\mathcal{C}\X}$: Since ${_\omega\X}\subseteq\omega^\bot$, it suffices to show that ${_\omega\X}\subseteq{_\mathcal{C}\X}$. Let $M\in{_\omega\X}$. Choose an $\E$-triangle sequence $\xymatrix@C=0.5cm{\cdots\ar[r]&W_1\ar[r]^{d_1}&W_0\ar[r]^{d_0}&M\ar@{-->}[r]&}$ as in (\ref{E}) with $T_i\in \omega$ and  $K_i\in{_\omega\X}$. Applying Lemma \ref{ch} to the $\E$-triangle $\xymatrix@C=0.5cm{K_1\ar[r]&W_0\ar[r]&M\ar@{-->}[r]&}$, we obtain two $\E$-triangles $$\xymatrix@C=0.5cm{H_0\ar[r]&M_0\ar[r]&K_1\ar@{-->}[r]&}$$ and $$\xymatrix@C=0.5cm{M_0\ar[r]&C_0\ar[r]&M\ar@{-->}[r]&}$$ with $H_0\in{_\mathcal{C}\X}$ and $C_0\in\mathcal{C}$. Since $K_1\in\omega^\bot\subseteq\mathcal{C}^\bot$ and $H_0\in\mathcal{C}^\bot$, we have $M_0\in\mathcal{C}^\bot$. Consider the following commutative diagram
\[\xymatrix@C=20pt@R=20pt{&K_2\ar@{=}[r]\ar[d]&K_2\ar[d]&\\
  H_0\ar[r]\ar@{=}[d]&Y\ar[r]\ar[d]&T_1\ar@{-->}[r]\ar[d]&\\
 H_0\ar[r]&M_0\ar[r]\ar@{-->}[d]&K_1\ar@{-->}[r]\ar@{-->}[d]&\\
 &&&
 }\]
 Since $T_1,H_0\in{_\mathcal{C}\X}$, we have $Y\in{_\mathcal{C}\X}$. Choose an $\E$-triangle $\xymatrix@C=0.5cm{H_1\ar[r]&C_1\ar[r]&Y\ar@{-->}[r]&}$ with $H_1\in{_\mathcal{C}\X}$ and $C_1\in\mathcal{C}$. Consider the following commutative diagram
 $$\xymatrix@C=20pt@R=20pt{H_1\ar@{=}[r]\ar[d]&H_1\ar[d]&\\
M_1\ar[r]\ar[d]&C_1\ar[d]\ar[r]&M_0\ar@{=}[d]\ar@{-->}[r]&\\
K_2\ar[r]\ar@{-->}[d]&Y\ar[r]\ar@{-->}[d]&M_0\ar@{-->}[r]&\\
&&&}$$
Since $H_1\in\mathcal{C}^\bot$ and $K_2\in \omega^\bot\subseteq\mathcal{C}^\bot$, we have $M_1\in\mathcal{C}^\bot$. By repeating the process to $M_1$, and so on, we can obtain an $\E$-triangle sequence
$\xymatrix@C=0.5cm{\cdots\ar[r]&C_1\ar[r]^{g_1}&C_0\ar[r]^{g_0}&M\ar@{-->}[r]&}$ with $C_i\in \mathcal{C}$ and $M_i\in{\mathcal{C}^\bot}$. Therefore, $M\in{_{\mathcal{C}}\X}$.
\end{proof}

\subsection{Bazzoni characterization of $n$-tilting pairs}

In what follows, we will give a Bazzoni characterization of $n$-tilting pairs. Before doing it, we first need to introduce two kinds of subcategories as follows.

Let $\mathcal T_{1}, \mathcal T_{2}$ are subcategories of $\C$. We define ${\rm Pres}^{n}_{\mathcal T_{2}}(\mathcal T_{1})$ (${\rm Copres}^{n}_{\mathcal T_{2}}(\mathcal T_{1})$, resp.) to be the subcategory consisting of each object $X\in\C$ such that there is an $\E$-triangle sequence:
$$X'\rightarrow T_{n}\rightarrow T_{n-1}\cdots\rightarrow T_{2}\rightarrow T_{1}\rightarrow X$$
$$(X\rightarrow T_{1}\rightarrow T_{2}\cdots\rightarrow T_{n-1}\rightarrow T_{n}\rightarrow X', \mbox{resp}.)$$
in $\C$ with $T_{i}\in\mathcal T_{1}, X'\in\mathcal T_{2}.$ It is obvious that $\mathcal{T}_1\subseteq {\rm Pres}^{n}_{\mathcal T_{2}}(\mathcal T_{1})$ ($\mathcal {T}_1\subseteq {\rm Copres}^{n}_{\mathcal T_{2}}(\mathcal T_{1})$, resp.). If $\mathcal T_{2}=\C$, we write ${\rm Pres}^{n}_{\mathcal T_{2}}(\mathcal T_{1})$ (${\rm Copres}^{n}_{\mathcal T_{2}}(\mathcal T_{1})$, resp.) as ${\rm Pres}^{n}(\mathcal T_{1})$ (${\rm Copres}^{n}(\mathcal T_{1})$, resp.).

Using this notion, we can give an equivalent description for ${_\mathcal{T}\X}$, that is,

\begin{lemma}\label{1}
Assume that $(\mathcal{C},\mathcal{T})$ is an $n$-tilting pair with $\mathcal{T}$ strongly contravariantly finite. The following are equivalent for an object $M\in\C$.  %$\mathcal{T}^\bot\cap{_\mathcal{C}\X}={_\mathcal{T}\X}$
\begin{itemize}
  \item [$(1)$] $M\in{_\mathcal{T}\X}$.
  \item [$(2)$] $M\in\mathcal{T}^\bot\cap{_\mathcal{C}\X}$.
  \item [$(3)$] $M\in {\rm Pres}^{n}_{_\mathcal{C}\X}(\mathcal T)$.
\end{itemize}
\end{lemma}
\begin{proof}
The equivalence of $(1)$ and $(2)$ has been proved in Lemma \ref{4.2}. Now to show that $(2)$ and $(3)$ are equivalent.

Let $M\in\mathcal{T}^\bot\cap{_\mathcal{C}\X}$. Using the same argument as that in Lemma \ref{4.2}, we have an $\E$-triangle $\xymatrix@C=0.5cm{M'\ar[r]&T_M\ar[r]&M\ar@{-->}[r]&}$ with $M'\in\omega^\bot$ and $T_M\in\omega$. Since $\mathcal{T}^\bot\subseteq\mathcal{C}^\bot$ and $\mathcal T\subseteq\mathcal{T}^\bot\cap{_\mathcal{C}\X}$, Using Lemma \ref{ort}(4), one can see that $M'\in\mathcal{T}^\bot\cap{_\mathcal{C}\X}$. Now replacing $M$ with $M'$ and repeating the above process, we have $M\in {\rm Pres}^{n}_{_\mathcal{C}\X}(\mathcal T)$.

Let $M\in {\rm Pres}^{n}_{_\mathcal{C}\X}(\mathcal T)$. There is an $\E$-triangle sequence $N\rightarrow T_{n}\rightarrow T_{n-1}\cdots\rightarrow T_{2}\rightarrow T_{1}\rightarrow M$ with $N\in {_\mathcal{C}\X}$ and $T_{i}\in\mathcal T$. Using Lemma \ref{ort}(3), we have $M\in{_\mathcal{C}\X}$. Now to show $M\in\mathcal{T}^\bot$. As $\mathcal T$ is self-orthogonal, one can see that $\E^{i}(\mathcal T,M)=\E^{i+n}(\mathcal T,N)$ for all $i\geq1$. Since $(\mathcal{C},\mathcal{T})$ is an $n$-tilting pair, for any $T\in\mathcal T$, there is an $\E$-triangle sequence $C_{n}\rightarrow C_{n-1}\cdots\rightarrow C_{1}\rightarrow C_{0}\rightarrow T$. Then we have $\E^{i+n}(T,X)=\E^{i}(C_{n},N)=0$ for all $i\geq1$ since $N\in {_\mathcal{C}\X}\subseteq\mathcal{C}^\bot$. Therefore $\E^{i}(\mathcal T,M)=0$ for all $i\geq1$, that is, $M\in\mathcal{T}^\bot$.
\end{proof}

Clearly, if
${\rm Pres}^{n}_{_\mathcal{C}\X}(\mathcal T)=\mathcal{T}^\bot\cap{_\mathcal{C}\X}$, then $\mathcal T\subseteq\mathcal{T}^\bot\cap{_\mathcal{C}\X}$, in particular, $\mathcal T$ is a self-othogonal subcategory.

%\begin{proof}
%Since $\mathcal T\subseteq {\rm Pres}^{n}_{_\mathcal{C}\X}(\mathcal T)$, it is obvious that the result holds.
%\end{proof}

\begin{proposition}\label{3}
If $\mathcal T$ is a strongly contravariantly finite subcategory in $\C$ and ${\rm Pres}^{n}_{_\mathcal{C}\X}(\mathcal T)=\mathcal{T}^\bot\cap{_\mathcal{C}\X}$, then ${\rm Pres}^{n}_{_\mathcal{C}\X}(\mathcal T)={\rm Pres}^{n}_{_\mathcal{C}\X}({\rm Pres}^{n}_{_\mathcal{C}\X}(\mathcal T))$.
\end{proposition}

\begin{proof}
Firstly, we show that ${\rm Pres}^{n}_{_\mathcal{C}\X}(\mathcal T)={\rm Pres}^{n+1}_{_\mathcal{C}\X}(\mathcal T)$. We just need to show that ${\rm Pres}^{n}_{_\mathcal{C}\X}(\mathcal T)\subseteq {\rm Pres}^{n+1}_{_\mathcal{C}\X}(\mathcal T)$. Indeed, for any $M\in {\rm Pres}^{n}_{_\mathcal{C}\X}(\mathcal T)$, there is an $\E$-triangle $\xymatrix@C=0.5cm{N\ar[r]&T_{1}\ar[r]&M\ar@{-->}[r]&}$ with $N\in {\rm Pres}^{n-1}_{_\mathcal{C}\X}(\mathcal T)$ and $T_{1}\in\mathcal T$. Using Lemma  \ref{ort}(3), we have $N\in{_\mathcal{C}\X}$. Since $M\in\mathcal{T}^\bot\cap {\rm Gen}(\mathcal T)$ and $\mathcal T$ is self-orthogonal, we have an $\E$-triangle $\xymatrix@C=0.5cm{L\ar[r]&T_{M}\ar[r]&M\ar@{-->}[r]&}$ with $L\in\mathcal{T}^\bot$ and $T_{M}\in\mathcal T$. Thus we have the following commutative diagram:
$$\xymatrix@C=20pt@R=20pt{&N\ar@{=}[r]\ar[d]&N\ar[d]&\\
L\ar[r]\ar@{=}[d]&Y\ar[r]\ar[d]&T_{1}\ar[d]\ar@{-->}[r]&\\
L\ar[r]&T_{M}\ar[r]\ar@{-->}[d]&M\ar@{-->}[r]\ar@{-->}[d]&\\
&&}$$
Because $T_{M}\in {_\mathcal{C}\X}$, $N\in {_\mathcal{C}\X}$ and $_\mathcal{C}\X$ is closed under extensions, we have $Y\in {_\mathcal{C}\X}$. Since $L\in \mathcal{T}^\bot$, one can see that $Y\cong L\oplus T_{1}$. Therefore $L\in \mathcal{T}^\bot\cap{_\mathcal{C}\X}={\rm Pres}^{n}_{_\mathcal{C}\X}(\mathcal T)$, since $ _\mathcal{C}\X$ is closed under direct summands. From the argument above, one can see that $M\in {\rm Pres}^{n+1}_{_\mathcal{C}\X}(\mathcal T)$.

Now we will show that ${\rm Pres}^{n}_{_\mathcal{C}\X}(\mathcal T)={\rm Pres}^{n}_{_\mathcal{C}\X}({\rm Pres}^{n}_{_\mathcal{C}\X}(\mathcal T))$. It is enough to show that $${\rm Pres}^{n}_{_\mathcal{C}\X}({\rm Pres}^{n}_{_\mathcal{C}\X}(\mathcal T))\subseteq {\rm Pres}^{n}_{_\mathcal{C}\X}(\mathcal T).$$ Let $M\in {\rm Pres}^{n}_{_\mathcal{C}\X}({\rm Pres}^{n}_{_\mathcal{C}\X}(\mathcal T))$. Then there is an $\E$-triangle sequence $X\rightarrow P_{n}\rightarrow P_{n-1}\cdots\rightarrow P_{2}\rightarrow P_{1}\rightarrow M$ with $P_{i}\in {\rm Pres}^{n}_{_\mathcal{C}\X}(\mathcal T)$ and $X\in {_\mathcal{C}\X}$. Let $\Y={\rm Pres}^{n}_{_\mathcal{C}\X}(\mathcal T)$, using Lemma \ref{ch} and ${\rm Pres}^{n}_{_\mathcal{C}\X}(\mathcal T)={\rm Pres}^{n+1}_{_\mathcal{C}\X}(\mathcal T)$, we have an
$\E$-triangle $\xymatrix@C=0.5cm{U\ar[r]&V\ar[r]&X\ar@{-->}[r]&}$ such that $U\in {\rm Pres}^{n}_{_\mathcal{C}\X}(\mathcal T)$ and there is an $\E$-triangle sequence $V\rightarrow T_{n}\rightarrow T_{n-1}\cdots\rightarrow T_{2}\rightarrow T_{1}\rightarrow M$ with $T_{i}\in\mathcal T$. We have $V\in {_\mathcal{C}\X}$ since $U,X\in {_\mathcal{C}\X}$ and ${_\mathcal{C}\X}$ is closed under extensions. Therefore, $M\in {\rm Pres}^{n}_{_\mathcal{C}\X}(\mathcal T)$.
\end{proof}

\begin{proposition}\label{4}
Assume that $\mathcal T$ is a strongly contravariantly finite subcategory in $\C$ and ${\rm Pres}^{n}_{_\mathcal{C}\X}(\mathcal T)=\mathcal{T}^\bot\cap{_\mathcal{C}\X}$. If $\mathcal C\subseteq {\rm Copres}^{n}({\rm Pres}^{n}_{_\mathcal{C}\X}(\mathcal T))$, then $\mathcal C\subseteq\check{\mathcal{T}_{n}}$.
\end{proposition}

\begin{proof}
If $\mathcal C\subseteq {\rm Copres}^{n}({\rm Pres}^{n}_{_\mathcal{C}\X}(\mathcal T))$, then for any $C\in\mathcal C$, there is an $\E$-triangle sequence $C\rightarrow P_{n}\rightarrow P_{n-1}\cdots\rightarrow P_{2}\rightarrow P_{1}\rightarrow P_{0}$ with $P_{i}\in {\rm Pres}^{n}_{_\mathcal{C}\X}(\mathcal T)$ for $1\leq i\leq n$. Using Proposition \ref{3}, one can see $P_{0}\in {\rm Pres}^{n}_{_\mathcal{C}\X}(\mathcal T)$. Let $\X={\rm Pres}^{n}_{_\mathcal{C}\X}(\mathcal T)$ and applying Lemma  \ref{ch}, we have an $\E$-triangle $\xymatrix@C=0.5cm{U\ar[r]&V\ar[r]&C\ar@{-->}[r]&}$ such that $U\in {\rm Pres}^{n}_{_\mathcal{C}\X}(\mathcal T)$ and $V\in\check{\mathcal{T}_{n}}$. Since $U\in {\rm Pres}^{n}_{_\mathcal{C}\X}(\mathcal T)=\mathcal{T}^\bot\cap{_\mathcal{C}\X}\subseteq\mathcal{C}^\bot$, one can see $V\cong C\oplus U.$ Using Lemma \ref{ort}(2), we have $C\in\check{\mathcal{T}_{n}}$.
\end{proof}

\begin{proposition}\label{5}
Assume that $\mathcal T$ is a strongly contravariantly finite subcategory in $\C$ and ${\rm Pres}^{n}_{_\mathcal{C}\X}(\mathcal T)=\mathcal{T}^\bot\cap{_\mathcal{C}\X}$. If $\mathcal C\subseteq {\rm Copres}^{n}({\rm Pres}^{n}_{_\mathcal{C}\X}(\mathcal T))$, then $\mathcal T\subseteq\hat{\mathcal{C}}_{n}$.
\end{proposition}

\begin{proof}
Since $\mathcal T\subset {\rm Pres}^{n}_{_\mathcal{C}\X}(\mathcal T)\subseteq{_\mathcal{C}\X}$, for any $T\in\mathcal T$, there is an $\E$-triangle sequence
$$\xymatrix@C=0.5cm{X_{n+1}\ar[r]^{f_{n+1}}&C_{n}\ar[r]^{f_n}&C_{n-1}\ar[r]\cdots\ar[r]&C_1\ar[r]^{f_1}&C_0\ar[r]^{f_0}&T&}$$
with $X_{n+1}\in {_\mathcal{C}\X}$ and $C_{i}\in\mathcal C$. Then we have $\E^{1}({\rm Cone}(f_{n+1}),X_{n+1})=\E^{n+1}(T,X_{n+1})$. If we have shown that $\E^{n+1}(T,X_{n+1})=0$, then ${\rm Cone}(f_{n+1})\in\mathcal C$ since $\mathcal C$ is closed under direct summands. Therefore $T\in\hat{\mathcal{C}}_{n}.$

Now to show that $\E^{n+1}(T,M)=0$ for any $T\in\mathcal T, M\in {_\mathcal{C}\X}$. Since $M\in {_\mathcal{C}\X}$, there is an $\E$-triangle sequence $Z\rightarrow C_{n}\rightarrow C_{n-1}\cdots\rightarrow C_{2}\rightarrow C_{1}\rightarrow M$ with $C_{i}\in\mathcal C$ and $Z\in {_\mathcal{C}\X}$.
Using Proposition \ref{4}, one can see that $C_{i}\in \check{\mathcal{T}_{n}}$. Applying the dual of Lemma \ref{ch}, we have an $\E$-triangle $\xymatrix@C=0.5cm{M\ar[r]&V\ar[r]&U\ar@{-->}[r]&}$ such that $U\in \check{\mathcal{T}}_{n-1}$ and there is an $\E$-triangle sequence $Z\rightarrow T_{n}\rightarrow T_{n-1}\cdots\rightarrow T_{2}\rightarrow T_{1}\rightarrow V$ with $T_{i}\in\mathcal T$. It is obvious that $V\in {\rm Pres}^{n}_{_\mathcal{C}\X}(\mathcal T)=\mathcal{T}^\bot\cap{_\mathcal{C}\X}$. Using the fact that $\check{\omega}_{n-1}=\{X\in X_{\omega}|\E^{n}(Y,X)=0, \mbox{ for all }\ Y\in\sideset{^{\bot}}{}{\mathop{\omega}}\}$ in the proof of Lemma \ref{ort}(2), we have $\E^{n+1}(T,M)=\E^{n}(T,U)=0$. If we let $M=\X_{n+1}$, then $\E^{n+1}(T,X_{n+1})=0$.
\end{proof}

Now we have the following Bazzoni characterization of $n$-tilting pairs in an extriangulated category.
\begin{theorem}\label{6}
Assume that $\mathcal C\subseteq {\rm Copres}^{n}({\rm Pres}^{n}_{_\mathcal{C}\X}(\mathcal T))$ and $\mathcal T$ is strongly contravariantly finite. Then $(\mathcal{C},\mathcal{T})$ is an $n$-tilting pair if and only if ${\rm Pres}^{n}_{_\mathcal{C}\X}(\mathcal T)=\mathcal{T}^\bot\cap{_\mathcal{C}\X}$.
\end{theorem}

\begin{proof}
Using Lemma \ref{1}, Propositions \ref{4} and  \ref{5}, one can see the result holds.
\end{proof}

In particular, using Theorem \ref{6} to Example \ref{exampair}(5), we get the Bazzoni characterization of an $n$-tilting subcategory in \cite{ZZhuang}.

\begin{corollary}{\rm (\cite[Theorem 1]{ZZhuang})}\label{7}
Assume that  $\mathcal T$ is a subcategory of $\C$ which is strongly contravariantly finite and closed under direct summands. Then $\mathcal T$ is an $n$-tilting subcategory if and only if ${\rm Pres}^{n}(\mathcal T)=\mathcal T^{\bot}.$
\end{corollary}

\begin{example}
If $\C=\mbox{mod}A$, where $A$ is an Artin algebra and $\mbox{mod}A$ is the category of finitely generated left $A$-modules. We can get the Bazzoni characterization of $n$-tilting pairs in \cite[Theorem 3.10]{wx} from Theorem \ref{6}.
\end{example}

\section{Auslander-Reiten correspondence}

In this section, we mainly study the Auslander-Reiten correspondence for tilting pairs, which classifies finite $\mathcal{C}$-tilting subcategories for a certain self-orthogonal subcategory $\mathcal{C}$ with some assumptions.

\begin{proposition}\label{4.3}
Assume that $\mathcal{T}$ is  a finite $\mathcal{C}$-tilting subcategory.
\begin{itemize}
  \item [$(1)$] ${_\mathcal{C}\X}=\check{{_\mathcal{T}\X}}\cap\mathcal{C}^\bot$.
  \item  [$(2)$] $\check{{_\mathcal{T}\X}}={_\mathcal{C}\check{\X}}$.
\end{itemize}
\end{proposition}

\begin{proof}
(1) ${_\mathcal{C}\X}\subseteq\check{{_\mathcal{T}\X}}\cap\mathcal{C}^\bot$: Since ${_\mathcal{C}\X}\subseteq\mathcal{C}^\bot$, it suffices to show that ${_\mathcal{C}\X}\subseteq\check{{_\mathcal{T}\X}}$. Let $M\in {_\mathcal{C}\X}$. Take an $\E$-triangle sequence $\xymatrix@C=0.4cm{M_n\ar[r]&C_n\ar[r]&\cdots\ar[r]&C_1\ar[r]&M}$ with $M_n\in {_\mathcal{C}\X}$ and each $C_i\in\mathcal{C}$. By the dual of Lemma \ref{ch}, there are an $\E$-triangle $\xymatrix@C=0.5cm{M\ar[r]&U\ar[r]&V\ar@{-->}[r]&}$ and an $\E$-triangle sequence $$\xymatrix@C=0.4cm{M_n\ar[r]&T_n\ar[r]&\cdots\ar[r]&T_1\ar[r]&V}$$ with $U\in\check{\mathcal{T}}$ and $T_i\in\mathcal{T}$.
Since $\mathcal{T}\subseteq\hat{\mathcal{C}}\subseteq{_\mathcal{C}\X}$, each $T_i\in{_\mathcal{C}\X}$. Moreover, $M_n\in{_\mathcal{C}\X}$, so $V\in{_\mathcal{C}\X}$ by Lemma \ref{ort}(3).

Claim: $V\in\mathcal{T}^\bot$. In fact, for the $\E$-triangle sequence $$\xymatrix@C=0.5cm{M_n\ar[r]&T_n\ar[r]&\cdots\ar[r]&T_1\ar[r]&V},$$ since $\mathcal{T}$ is self-orthogonal, we have $\E^i(T,V)\cong\E^{i+n}(T,M_n)$ for any $T\in\mathcal{T}$ and any $i\geq 1$.
Since $\mathcal{T}$ is an $n$-$\mathcal{C}$-tilting subcategory, there is an $\E$-triangle sequence $\xymatrix@C=0.5cm{C_n\ar[r]&\cdots\ar[r]&C_1\ar[r]&C_0\ar[r]&T}$ with each $C_i\in\mathcal{C}$. The condition $M_n\in{_\mathcal{C}\X}\subseteq\mathcal{C}^\bot$ implies $\E^{i+n}(T,M_n)\cong\E^i(C_n,M_n)=0$, and hence $\E^i(T,V)=0$ for  any $i\geq 1$. Thus $V\in\mathcal{T}^\bot$. By Lemma \ref{4.2}, $V\in {_\mathcal{T}\X}$, and hence $M\in\check{{_\mathcal{T}\X}}$ by Proposition \ref{equ}.

$\check{{_\mathcal{T}\X}}\cap\mathcal{C}^\bot\subseteq{_\mathcal{C}\X}$: Let $M\in\check{{_\mathcal{T}\X}}\cap\mathcal{C}^\bot$. Choose an $\E$-triangle sequence
$$\xymatrix@C=0.5cm{M\ar[r]^{f_m}&X_m\ar[r]&\cdots\ar[r]^{f_0}&X_0}$$ with each $X_i\in{_\mathcal{T}\X}$, and consider the corresponding $\E$-triangles
$$\xymatrix@C=0.5cm{K_{i+1}\ar[r]&X_i\ar[r]&K_i\ar@{-->}[r]&}$$ for $1\leq i\leq m$ ($K_{m+1}=M$, $K_1=X_0$). By Lemma \ref{4.2}, ${_\mathcal{T}\X}\subseteq{_\mathcal{C}\X}\subseteq\mathcal{C}^\bot$, and so $X_i\in\mathcal{C}^\bot$. Moreover, $M\in\mathcal{C}^\bot$ implies $K_i\in\mathcal{C}^\bot$. By Lemma \ref{ort}, we can iteratively obtain $K_2\in{_\mathcal{C}\X}$, $K_3\in{_\mathcal{C}\X}$, $\cdots$, $K_m\in{_\mathcal{C}\X}$, $M\in{_\mathcal{C}\X}$.

(2) By (1), ${_\mathcal{C}\X}\subseteq\check{{_\mathcal{T}\X}}$. By Lemma \ref{ort}, ${_\mathcal{C}\check{\X}}\subseteq\check{{_\mathcal{T}\X}}$. Conversely, $\check{{_\mathcal{T}\X}}\subseteq{_\mathcal{C}\check{\X}}$.
\end{proof}

\begin{lemma}\label{4.5} {\rm(\cite[Lemma 2.3]{CZZ})}
Let $\omega$ be closed under extensions.
\begin{itemize}
  \item [$(1)$] If $f: W\to M$ is a right $\omega$-approximation of $M$, then ${\rm CoCone}f\in\omega^{\bot_1}$.
  \item [$(2)$] If $f: M\to W$ is a left $\omega$-approximation of $M$, then ${\rm Cone}f\in{^{\bot_1}\omega}$.
\end{itemize}
\end{lemma}

\begin{definition}
Let $\mathcal{S}$ be a finite subset of objects in $\C$, and $M\in\C$.
A \emph{finite $\mathcal{S}$-filtration} of $M$ is a class of $\E$-triangles
\begin{gather*}
  \xymatrix@C=0.5cm{0\ar[r]&M_1\ar[r]&S_1\ar@{-->}[r]&} \\
  \xymatrix@C=0.5cm{M_1\ar[r]&M_2\ar[r]&S_2\ar@{-->}[r]&}\\
  \vdots\\
  \xymatrix@C=0.5cm{M_{n-2}\ar[r]&M_{n-1}\ar[r]&S_{n-1}\ar@{-->}[r]&} \\
  \xymatrix@C=0.5cm{M_{n-1}\ar[r]&M\ar[r]&S_n\ar@{-->}[r]&}
\end{gather*}
such that each $S_i\in\mathcal{S}$ for any $1\leq i\leq n$.

A category $\C$ is called \emph{finitely filtered} if there is a finite subset $\mathcal{S}$ in $\C$ such that each object in $\C$ has a finite $\mathcal{S}$-filtration.
\end{definition}

\begin{lemma}\label{4.7}
Let $\mathcal{C}$ be self-orthogonal with $\check{{_\mathcal{C}\X}}$ finitely filtered. Assume that $\mathcal{D}\subseteq{_\mathcal{C}\X}$ is closed under extensions and cones of inflations with $\check{\mathcal{D}}=\check{{_\mathcal{C}\X}}$. Then
\begin{itemize}
  \item [$(1)$] there is some $n$ with $\check{\mathcal{D}}=\check{\mathcal{D}}_n$.
  \item [$(2)$] ${^\bot\mathcal{D}}\cap{_\mathcal{C}{\X}}\subseteq \hat{{\mathcal{C}}}$.
\end{itemize}
\end{lemma}

\begin{proof}
(1) By assumption, there is a finite set $\mathcal{S}$ such that each object in  $\check{\mathcal{D}}$ has a finite filtration by $\mathcal{S}$. Set $n=\max\{\mathcal{D}\mbox{-codim}S, S\in\mathcal{S}\}$. Then by Lemma \ref{3.6}, $\check{\mathcal{D}}=\check{\mathcal{D}}_n$.

(2) Let $X\in{^\bot\mathcal{D}}\cap{_\mathcal{C}{\X}}$. Since $X\in{_\mathcal{C}{\X}}$, there is an $\E$-triangle sequence
$$\xymatrix@C=0.5cm{K_{n+1}\ar[r]&C_n\ar[r]^{f_n}&\cdots\ar[r]^{f_1}&C_0\ar[r]^{f_0}&X}$$ with $C_i\in\mathcal{C}$ and $K_i\in{_\mathcal{C}{\X}}\subseteq\mathcal{C}^\bot$. Here each $K_i$ arises in the corresponding $\E$-triangle $\xymatrix@C=0.5cm{K_{i+1}\ar[r]&C_i\ar[r]&K_i\ar@{-->}[r]&}$ ($K_0=X$). It follows that
\begin{align*}
   \E(K_n,K_{n+1})&\cong\E^2(K_{n-1},K_{n+1})  \\
   & \cong\cdots\\
   &\cong\E^{n+1}(X,K_{n+1}).
\end{align*}
Moreover, since $K_{n+1}\in\check{{_\mathcal{C}\X}}=\check{\mathcal{D}}=\check{\mathcal{D}}_n$ by (1), then $$\E(K_n,K_{n+1})\cong\E^{n+1}(X,K_{n+1})=0$$ and hence the $\E$-triangle $\xymatrix@C=0.5cm{K_{n+1}\ar[r]&C_n\ar[r]&K_n\ar@{-->}[r]&}$ is split. Thus $K_{n+1}\in\mathcal{C}$ and, therefore, $X\in \hat{\mathcal{C}}$.
\end{proof}

Let $\Y$ be a subcategory of $\C$. If $\mathcal I\subseteq\Y\cap\Y^\bot$, and for any $Y\in\Y$, there is an $\E$-triangle $\xymatrix@C=0.5cm{Y\ar[r]&I_{Y}\ar[r]&Y'\ar@{-->}[r]&}$ with $I_{Y}\in\mathcal I$ and $Y'\in\Y$, we say that $\mathcal I$ is a \emph{relative injective cogenerator} of $\Y$. Dually, one can define the relative projective generator of $\Y$.

Let $\mathcal D\subseteq \mathcal C$ be subcategories of $\C$. We say that $\mathcal D$ is a \emph{relative resolving subcategory} of $\mathcal C$ if $\mathcal D$ is closed under extensions, direct summands and cocones of deflations and contains ${\rm Proj}(\mathcal C)$ (where ${\rm Proj}(\mathcal C)$ is the subcategory of projective objects in $\mathcal C$). Dually, one can define the relative coresolving subcategory of $\mathcal C$. In particular, if $\mathcal C=\C$ and $\mathcal D$ is a relative resolving (coresolving resp.) subcategory of $\mathcal C$, we say that $\mathcal D$ is a resolving (coresolving resp.) subcategory of $\C$.

\begin{lemma}\label{4.8}
Let $\mathcal{C}$ be self-orthogonal such that ${_\mathcal{C}\X}$ has a relative injective cogenerator $\mathcal{I}$. If $\mathcal{D}\subseteq{_\mathcal{C}\X}$ is relative coresolving in ${_\mathcal{C}\X}$, then  $\check{{_\mathcal{C}\X}}\cap{^{\bot_1}\mathcal{D}}\subseteq{^\bot\mathcal{D}}$.
\end{lemma}

\begin{proof}
Let $D\in\mathcal{D}\subseteq{_\mathcal{C}\X}$. There is an $\E$-triangle sequence
$\xymatrix@C=0.5cm{D\ar[r]&I_0\ar[r]^{f_0}&I_1\ar[r]^{f_1}&\cdots}$ with $I_i\in\mathcal{I}$. Consider the corresponding $\E$-triangles
$\xymatrix@C=0.5cm{K_i\ar[r]&I_i\ar[r]&K_{i+1}\ar@{-->}[r]&}$ ($K_0=D$). Since $\mathcal{D}$ is relative coresolving, each $K_i\in\mathcal{D}$.

Let $X\in\check{{_\mathcal{C}\X}}\cap{^{\bot_1}\mathcal{D}}$. Then $\E^i(X,D)\cong\E(X,K_{i-1})=0$ for any $i\geq 1$, as desired.
\end{proof}

In the following, we give a sufficient condition such that $\mathcal{D}\cap{^\bot\mathcal{D}}$ is a generator for $\mathcal{D}$.

\begin{lemma}\label{4.9}
Let $\mathcal{C}$ be self-orthogonal such that ${_\mathcal{C}\X}$ has a relative injective cogenerator, and $\mathcal{D}$ be relative coresolving and strongly covariantly finite in ${_\mathcal{C}\X}$. Then for any $D\in\mathcal{D}$, there is an $\E$-triangle $\xymatrix@C=0.5cm{D'\ar[r]&W_D\ar[r]&D\ar@{-->}[r]&}$ with $D'\in\mathcal{D}$ and $W_D\in\omega$, where $\omega=\mathcal{D}\cap{^\bot\mathcal{D}}$.

In particular, $\mathcal{D}\subseteq{_\omega\X}$.
\end{lemma}

\begin{proof}
For any $D\in\mathcal{D}\subseteq{_\mathcal{C}\X}$, we can take an $\E$-triangle $\xymatrix@C=0.5cm{L\ar[r]&C_D\ar[r]&D\ar@{-->}[r]&}$ with $C_D\in\mathcal{C}$ and $L\in{_\mathcal{C}\X}$. For $L$, there is an $\E$-triangle $\xymatrix@C=0.5cm{L\ar[r]&D'\ar[r]&M\ar@{-->}[r]&}$ with $D'\in\mathcal{D}$ and $M\in{^{\bot_1}\mathcal{D}}$. By Lemma \ref{4.8}, $M\in{^{\bot}\mathcal{D}}$. Consider the following commutative diagram
\[\xymatrix@C=20pt@R=20pt{L\ar[r]\ar[d]&C_D\ar[r]\ar[d]&D\ar@{-->}[r]\ar@{=}[d]&\\
 D'\ar[r]\ar[d]&W_D\ar[r]\ar[d]&D\ar@{-->}[r]&\\
 M\ar@{=}[r]\ar@{-->}[d]&M\ar@{-->}[d]&&\\
 &&&
 }\]
By the middle row, $W_D\in\mathcal{D}$. Moreover, $C_D\in\mathcal{C}\subseteq {^{\bot}\mathcal{D}}$ and $M\in{^{\bot}\mathcal{D}}$, we have $W_D\in{^{\bot}\mathcal{D}}$, and hence $W_D\in\omega$.
\end{proof}

\begin{lemma}\label{4.10}
Let $\mathcal{C}$ be self-orthogonal such that ${_\mathcal{C}\X}$ has a relative injective cogenerator. Assume that $\mathcal{D}$ is relative coresolving and strongly covariantly finite in ${_\mathcal{C}\X}$ with $\check{\mathcal{D}}=\check{{_\mathcal{C}\X}}$. Then ${^{\bot}\mathcal{D}}\cap{_\mathcal{C}\X}\subseteq\check{\omega}$,  where $\omega=\mathcal{D}\cap{^\bot\mathcal{D}}$. In particular, $\mathcal{C}\subseteq\check{\omega}$.
\end{lemma}

\begin{proof}
Let $X\in{_\mathcal{C}\X}$. Since $\mathcal{D}$ is strongly covariantly finite in ${_\mathcal{C}\X}$, there is an $\E$-triangle $\xymatrix@C=0.5cm{X\ar[r]&D_0\ar[r]&X_1\ar@{-->}[r]&}$ with $D_0\in\mathcal{D}$, $X_1\in{_\mathcal{C}\X}$. By Lemma \ref{4.5}, we may take $X_1\in{^{\bot_1}\mathcal{D}}$. By Lemma \ref{4.8},  $X_1\in{^{\bot}\mathcal{D}}$. In this case, if $X\in{^{\bot}\mathcal{D}}$, then $D_0\in{^{\bot}\mathcal{D}}$, and hence $D_0\in\mathcal{D}\cap{^\bot\mathcal{D}}=\omega$.
By repeating this process to $X_1$, and so on, we obtain an $\E$-triangle sequence
$\xymatrix@C=0.5cm{X\ar[r]^{f_0}&D_0\ar[r]^{f_1}&D_1\ar[r]&\cdots\ar[r]&}$, where each $D_i\in\omega$ and $X_i\in {^\bot\mathcal{D}}\cap{_\mathcal{C}\X}$. Since ${_\mathcal{C}\X}\subseteq\check{{_\mathcal{C}\X}}=\check{\mathcal{D}}$, there is some $m$ such that $X_m\in\mathcal{D}$. It follows that $X_m\in\mathcal{D}\cap{^\bot\mathcal{D}}=\omega$, which implies $X\in\check{\omega}$. Therefore, ${^{\bot}\mathcal{D}}\cap{_\mathcal{C}\X}\subseteq\check{\omega}$.
\end{proof}

Now if, in addition $\check{{_\mathcal{C}\X}}$ is finite filtered, then we can construct a $\mathcal{C}$-tilting subcategory by taking
${^\bot\mathcal{D}}\cap\mathcal{D}$.

\begin{proposition}\label{4.11}
Let $\mathcal{C}$ be self-orthogonal such that ${_\mathcal{C}\X}$ has a relative injective cogenerator and $\check{{_\mathcal{C}\X}}$ is finite filtered. Assume that $\mathcal{D}$ is relative coresolving and strongly covariantly finite in ${_\mathcal{C}\X}$ with $\check{\mathcal{D}}=\check{{_\mathcal{C}\X}}$. Then $\omega:={^\bot\mathcal{D}}\cap\mathcal{D}$ is a $\mathcal{C}$-tilting subcategory.
%Then $\omega={\rm add}T$ for a unique basic $C$-tilting object $T$, where $\omega={^\bot\mathcal{D}}\cap\mathcal{D}$.
In this case, $\mathcal{D}={_\omega\X}$.
\end{proposition}

\begin{proof}
Clearly, $\omega$ is self-orthogonal. Moreover, $\omega\subseteq\hat{\mathcal{C}}$ by Lemma \ref{4.7}, and $\mathcal{C}\subseteq\check{\omega}$ by Lemma \ref{4.10}. Thus the pair $(\mathcal{C},\omega)$ is a tilting pair.

By Lemma \ref{4.9}, $\mathcal{D}\subseteq{_\omega\X}$. On the other hand, ${_\omega\X}\subseteq{_C\X}\subseteq\check{{_C\X}}=\check{\mathcal{D}}=\check{\mathcal{D}}_n$
for some $n$. Now let $M\in {_\omega\X}$, by assumption, we can take an $\E$-triangle sequence
$$\xymatrix@C=0.5cm{K_{n+1}\ar[r]&W_n\ar[r]^{f_n}&\cdots \ar[r]&W_0\ar[r]^{f_0}&M}$$
with each $W_i\in\omega\subseteq\mathcal{D}$. Then $K_{n+1}\in\mathcal{D}$ and hence $M\in \mathcal{D}$ since $\mathcal{D}$ is closed under cones of inflations. This shows that ${_\omega\X}\subseteq\mathcal{D}$, and consequently $\mathcal{D}={_\omega\X}$.
\end{proof}

Now we show the first type of Auslander-Reiten correspondence, which gives a classification for
$\mathcal{C}$-tilting subcategories in terms of a certain class of relative coresolving strongly covariantly finite
subcategories in ${_\mathcal{C}\X}$.

\begin{theorem}\label{4.12}
Let $\mathcal{C}$ be self-orthogonal such that ${_\mathcal{C}\X}$ has a relative injective cogenerator and
$\check{{_\mathcal{C}\X}}$ is finite filtered. Then there is a one-one correspondence as follows:
 \begin{align*}
   \left\{\mbox{finite  }\mathcal{C}\mbox{-tilting subcategories}\right\} &\longrightarrow {\left\{ \begin{array}{c}
                                                                                             \mbox{subcategories }\mathcal{D}\mbox{ which are relative} \\
                                                                                             \mbox{coresolving, strongly covariantly}\\
                                                                                             \mbox{finite in }{_\mathcal{C}\X},\mbox{ and satisfies }\check{\mathcal{D}}=\check{{_\mathcal{C}\X}}
                                                                                           \end{array}\right\}}\\
   \mathcal{T}&\mapsto \   {_\mathcal{T}\X}=\mathcal{T}^\bot\cap{_\mathcal{C}\X}
 \end{align*}
 where the inverse map is given by $\mathcal{D} \mapsto
{^\bot\mathcal{D}}\cap\mathcal{D}$.
\end{theorem}

\begin{proof}
(i) Assume that $\mathcal{T}$ is a $\mathcal{C}$-tilting subcategory. Then $\check{_\mathcal{T}\X}=\check{{_\mathcal{C}\X}}$ by Proposition \ref{4.3}. Let $M\in{_\mathcal{C}\X}$. Then $M\in\check{_\mathcal{T}\X}$, and by Proposition \ref{equ}, there is an $\E$-triangle $\xymatrix@C=0.5cm{M\ar[r]&U'\ar[r]&V'\ar@{-->}[r]&}$ with $U'\in{_\mathcal{T}\X}$ and $V'\in\check{\mathcal{T}}$. By Lemma \ref{ort}, $\check{\mathcal{T}}\bot {_\mathcal{T}\X}$, and hence $M\to U'$ is a left ${_\mathcal{T}\X}$-approximation. That is, ${_\mathcal{T}\X}$ is strongly covariantly finite in ${_\mathcal{C}\X}$. It is easy to check that ${_\mathcal{T}\X}$ is relative coresolving in ${_\mathcal{C}\X}$.

(ii) Assume that $\mathcal{D}$ is a relative coresolving, strongly covariantly finite subcategory in ${_\mathcal{C}\X}$, and satisfies $\check{\mathcal{D}}=\check{{_\mathcal{C}\X}}$. By Proposition \ref{4.11},
${^\bot\mathcal{D}}\cap\mathcal{D}$ is a $\mathcal{C}$-tilting subcategory.

(iii) By Proposition \ref{4.11}, $_{{^\bot\mathcal{D}}\cap\mathcal{D}}\X=\mathcal{D}$. Moreover, let $M\in{^\bot({_\mathcal{T}\X})\cap{_\mathcal{T}\X}}$. Then there is an $\E$-triangle $\xymatrix@C=0.5cm{M'\ar[r]&T\ar[r]&M\ar@{-->}[r]&}$ with $T\in{\mathcal{T}}$ and $M'\in{_\mathcal{T}\X}$. It is split and hence $M\in\mathcal{T}$, which shows that ${^\bot({_\mathcal{T}\X})\cap{_\mathcal{T}\X}}\subseteq\mathcal{T}$. Moreover, it is easy to see $\mathcal{T}\subseteq{^\bot({_\mathcal{T}\X})\cap{_\mathcal{T}\X}}$, and  thus ${^\bot({_\mathcal{T}\X})\cap{_\mathcal{T}\X}}=\mathcal{T}$.
\end{proof}

Using Theorem \ref{4.12} to Example \ref{exampair}(5), we get the Auslander-Reiten correspondence of an $n$-tilting object in \cite{ZZhuang}.
\begin{corollary}{\rm (\cite[Theorem 2]{ZZhuang})}
Let $\C$ be a finite filtered extriangulated category. The assignments $T\mapsto T^{\bot}$ and $\X\mapsto \sideset{^{\bot}}{}{\mathop{\X}}\cap\X$ give a one-one correspondence between the class of $n$-tilting objects $T$ and coresolving strongly covariantly finite subcategories $\X,$ which are closed under direct summands with $\check{\X_{n}}=\C.$
\end{corollary}

In the rest of this paper, we will give the second type of Auslander-Reiten correspondence. To do it, we first need to build a correspondence between
relative coresolving strongly covariantly finite subcategories and relative resolving  strongly contravariantly finite subcategories in ${_\mathcal{C}\X}$.

\begin{lemma}\label{4.13}
Let $\mathcal{C}$ be self-orthogonal such that ${_\mathcal{C}\X}$ has a  relative injective cogenerator. Assume that $\mathcal{D}$ is relative coresolving and strongly covariantly finite in ${_\mathcal{C}\X}$. Then
\begin{itemize}
  \item [$(1)$] $\mathcal{D}$ is relative coresolving and strongly covariantly finite in $\check{{_\mathcal{C}\X}}$.
  \item [$(2)$] ${^\bot\mathcal{D}\cap{_\mathcal{C}\X}}$ is relative resolving and strongly contravariantly finite in ${_\mathcal{C}\X}$.
  \item [$(3)$] ${^\bot\mathcal{D}\cap\check{{_\mathcal{C}\X}}}$ is strongly contravariantly finite in $\check{{_\mathcal{C}\X}}$.
\end{itemize}
\end{lemma}

\begin{proof}
(1) Since the relative injective cogenerator of ${_\mathcal{C}\X}$ is also the relative injective cogenerator of $\check{{_\mathcal{C}\X}}$, it follows that $\mathcal{D}$ is relative coresolving in $\check{{_\mathcal{C}\X}}$.

Now let $X\in\check{{_\mathcal{C}\X}}$. By Proposition \ref{equ}, there is an $\E$-triangle $\xymatrix@C=0.5cm{X\ar[r]&Y\ar[r]&Z \ar@{-->}[r]&}$ with $Y\in{_C\X}$ and $Z\in\check{\mathcal{C}}$. For $Y$, by Lemmas \ref{4.5} and \ref{4.8}, there is an $\E$-triangle $\xymatrix@C=0.5cm{Y\ar[r]&D\ar[r]&M\ar@{-->}[r]&}$ with $D\in\mathcal{D}$ and $M\in{^\bot\mathcal{D}}$. Consider the following commutative diagram
$$\xymatrix@C=20pt@R=20pt{
X\ar[r]\ar@{=}[d]&Y\ar[r]\ar[d]&Z\ar[d]\ar@{-->}[r]&\\
X\ar[r]&D\ar[r]\ar[d]&N\ar@{-->}[r]\ar[d]&\\
&M\ar@{=}[r]\ar@{-->}[d]&M\ar@{-->}[d]&\\
&&&}$$
Since $\mathcal{D}\subseteq{_\mathcal{C}\X}\subseteq \mathcal{C}^\bot$, $\check{\mathcal{C}}\subseteq{^\bot\mathcal{D}}$. Thus $Z\in{^\bot\mathcal{D}}$. Moreover, $M\in{^\bot\mathcal{D}}$, we can get $N\in{^\bot\mathcal{D}}$, which shows that $X\to D$ is a left $\mathcal{D}$-approximation. Therefore, $\mathcal{D}$ is strongly covariantly finite in $\check{{_\mathcal{C}\X}}$.

(2) Clearly, ${^\bot\mathcal{D}\cap{_\mathcal{C}\X}}$ is relative resolving in ${_\mathcal{C}\X}$. Let $X\in{_\mathcal{C}\X}$. By Lemmas \ref{4.5} and \ref{4.8}, there is an $\E$-triangle $\xymatrix@C=0.5cm{X\ar[r]&D\ar[r]&M \ar@{-->}[r]&}$ with $D\in\mathcal{D}$ and $M\in{^\bot\mathcal{D}}$. Since $D\in\mathcal{D}\subseteq{_\mathcal{C}\X}$, $M\in{_\mathcal{C}\X}$ by Lemma \ref{ort}, and hence $M\in{^\bot\mathcal{D}}\cap{_\mathcal{C}\X}$. Moreover, by Lemma \ref{4.9}, there is an $\E$-triangle $\xymatrix@C=0.5cm{D'\ar[r]&W_D\ar[r]&D \ar@{-->}[r]&}$ with $D'\in\mathcal{D}$ and $W_D\in\omega=\mathcal{D}\cap{^\bot\mathcal{D}}$.
Consider the following commutative diagram
$$\xymatrix@C=20pt@R=20pt{D'\ar@{=}[r]\ar[d]&D'\ar[d]&\\
Y\ar[r]\ar[d]&W_D\ar[d]\ar[r]&M\ar@{=}[d]\ar@{-->}[r]&\\
X\ar[r]\ar@{-->}[d]&D\ar[r]\ar@{-->}[d]&M\ar@{-->}[r]&\\
&&&}$$
Since $D',X\in \mathcal{C}^\bot$, $Y\in C^\bot$. Moreover, since $W_D,M\in{_\mathcal{C}\X}$, we have $Y\in {_\mathcal{C}\X}$ by Lemma \ref{ort}. Finally, since $W_D,M\in {^\bot\mathcal{D}}$, we obtain $Y\in{^\bot\mathcal{D}}$, and consequently $Y\in{^\bot\mathcal{D}}\cap{_\mathcal{C}\X}$. This shows that $Y\to X$ is a right ${^\bot\mathcal{D}}\cap{_\mathcal{C}\X}$-approximation, which shows that ${^\bot\mathcal{D}}\cap{_\mathcal{C}\X}$ is strongly contravariantly finite in ${_\mathcal{C}\X}$.

(3) Similar to (2).
\end{proof}

By a similar argument to that of Lemmas \ref{4.7}, \ref{4.8} and \ref{4.13}, we can get

\begin{lemma}\label{4.14}
Let $\mathcal{C}$ be self-orthogonal and ${_\mathcal{C}\X}$ has a relative injective cogenerator. Assume that $\mathcal{B}$ is a relative resolving and strongly contravariantly finite in ${_\mathcal{C}\X}$.
\begin{itemize}
  \item [$(1)$] If $X\in{_\mathcal{C}\X}\cap\mathcal{B}^{\bot_1}$, the $X\in\mathcal{B}^{\bot}$.
  \item [$(2)$] For any $B\in \mathcal{B}$, there is an $\E$-triangle $\xymatrix@C=0.5cm{B\ar[r]&W_B\ar[r]&B' \ar@{-->}[r]&}$ with $W_B\in\omega=\mathcal{B}^\bot\cap\mathcal{B}$, and $B'\in\mathcal{B}$. In particular, $\mathcal{B}\subseteq{\X_\omega}$.
  \item [$(3)$] $\mathcal{B}^\bot\cap{_\mathcal{C}\X}$ is relative coresolving and strongly covariantly finite in ${_\mathcal{C}\X}$.
\end{itemize}
\end{lemma}

Now we have

\begin{proposition}\label{4.15}
Let $\mathcal{C}$ be self-orthogonal  and ${_\mathcal{C}\X}$ has a relative injective cogenerator.
 Then there is a one-one correspondence as follows:
 \begin{align*}
  \left\{\begin{array}{c}
            \mbox{relative coresolving,}\\\mbox{ strongly covariantly finite}  \\
             \mbox{subcategories in }{_\mathcal{C}\X}
          \end{array}\right\} &\longrightarrow \left\{\begin{array}{c}
            \mbox{relative resolving,}\\ \mbox{strongly contravariantly finite}  \\
             \mbox{subcategories in }{_\mathcal{C}\X}
          \end{array}\right\}  \\
   \mathcal{D}& \ \mapsto \  {^\bot\mathcal{D}}\cap{_\mathcal{C}\X}
 \end{align*}
 where the inverse map is given by $\mathcal{B}\mapsto {\mathcal{B}^\bot\cap{_\mathcal{C}\X}}$.
\end{proposition}

\begin{proof}
By Lemmas \ref{4.13} and \ref{4.14}, the maps above are well-defined.

$({^\bot\mathcal{D}\cap{_\mathcal{C}\X}})^\bot\cap{_\mathcal{C}\X}=\mathcal{D}$: Clearly, $\mathcal{D}\subseteq({^\bot\mathcal{D}\cap{_\mathcal{C}\X}})^\bot\cap{_\mathcal{C}\X}$. Conversely, let $X\in({^\bot\mathcal{D}\cap{_\mathcal{C}\X}})^\bot\cap{_\mathcal{C}\X}$. Since $\mathcal{D}$ is strongly covariantly finite, there is an $\E$-triangle $\xymatrix@C=0.5cm{X\ar[r]&D\ar[r]&Z \ar@{-->}[r]&}$ with $D\in\mathcal{D}$ and $Z\in{^\bot\mathcal{D}\cap{_\mathcal{C}\X}}$. It is split since $X\in({^\bot\mathcal{D}\cap{_\mathcal{C}\X}})^\bot$, which shows that $D\cong X\oplus Z$. Thus $X\in\mathcal{D}$, and hence $({^\bot\mathcal{D}\cap{_\mathcal{C}\X}})^\bot\cap{_\mathcal{C}\X}\subseteq\mathcal{D}$. Therefore, $({^\bot\mathcal{D}\cap{_\mathcal{C}\X}})^\bot\cap{_\mathcal{C}\X}=\mathcal{D}$.

Dually, we can prove ${^\bot(\mathcal{B}^\bot\cap{_\mathcal{C}\X})\cap{_\mathcal{C}\X}}=\mathcal{B}$.
\end{proof}

\begin{lemma}\label{4.16}
Let $C$ be self-orthogonal and $\check{{_\mathcal{C}\X}}$ be finitely filtered. Assume that $\mathcal{D}$ is strongly covariantly finite and relative coresolving in ${_\mathcal{C}\X}$. Then

there is some $n$ such that $\E^{n+1}({^\bot\mathcal{D}}\cap\check{{_\mathcal{C}\X}}, {_\mathcal{C}\X})=0$, provided $\check{\mathcal{D}}=\check{{_\mathcal{C}\X}}$.
\end{lemma}

\begin{proof}
 Let $\mathcal{S}=\{S_1,\cdots,S_t\}\subseteq{^\bot\mathcal{D}\cap\check{{_\mathcal{C}\X}}}$ be the finite subcategory. For each $S_i\in\mathcal{S}$, since $S_i\in\check{{_\mathcal{C}\X}}$, by Proposition \ref{equ} there is an $\E$-triangle $\xymatrix@C=0.5cm{S_i\ar[r]&X_i\ar[r]&Z_i \ar@{-->}[r]&}$ with $X_i\in{_\mathcal{C}\X}$ and $Z_i\in\check{\mathcal{C}}$. Since $\mathcal{D}\subseteq{_\mathcal{C}\X}\subseteq \mathcal{C}^\bot$, $\check{\mathcal{C}}\subseteq{^\bot\mathcal{D}}$, which shows that $Z_i\in{^\bot\mathcal{D}}$. Moreover, since $S_i\in{^\bot\mathcal{D}}$, $X_i\in{^\bot\mathcal{D}}$, and hence $X_i\in{^\bot\mathcal{D}}\cap{_C\X}$. By Lemma \ref{4.7}, ${^\bot\mathcal{D}}\cap{_\mathcal{C}\X}\subseteq \hat{\mathcal{C}}$. Denote $n=\max\{\mathcal{C}\mbox{-dim}X_i\mid 1\leq i\leq t\}$. Then $\E^{n+1}(S_i,{_\mathcal{C}\X})=0$, and hence $\E^{n+1}({^\bot\mathcal{D}}\cap\check{{_\mathcal{C}\X}},{_\mathcal{C}\X})=0$.
\end{proof}

\begin{proposition}\label{4.17}
Let $\mathcal{C}$ be self-orthogonal and $\check{{_\mathcal{C}\mathcal{X}}}$ be finitely filtered. Assume that $\mathcal{D}$ is strongly covariantly finite and relative coresolving in ${_\mathcal{C}\X}$. Then $\check{\mathcal{D}}=\check{{_\mathcal{C}\X}}$ if and only if ${^\bot\mathcal{D}}\cap{_\mathcal{C}\X}\subseteq \hat{\mathcal{C}}$.
\end{proposition}

\begin{proof}
The ``only if" part follows from Lemma \ref{4.7}.

The ``if" part. Since $\mathcal{D}\subseteq{_\mathcal{C}\X}$, $\check{\mathcal{D}}\subseteq\check{{_\mathcal{C}\X}}$.
On the other hand, let $X\in\check{{_\mathcal{C}\X}}$. Then there is an $\E$-triangle sequence $\xymatrix@C=0.5cm{X\ar[r]^{f_0}&D_0\ar[r]^{f_1}&D_1 \ar[r]&\cdots}$ with each $D_i\in\mathcal{D}$ and each $K_i\in{^\bot\mathcal{D}}\cap\check{{_\mathcal{C}\X}}$. Here $K_i$ arise in the corresponding $\E$-triangle $\xymatrix@C=0.5cm{K_{i-1}\ar[r]&D_i\ar[r]&K_i\ar@{-->}[r]&}$.
 By Lemma \ref{3.6}, $K_m\in{_\mathcal{C}\X}$ for some $m$, and thus $K_m\in{^\bot\mathcal{D}}\cap{_\mathcal{C}\X}\subseteq\hat{\mathcal{C}}$. By Lemma \ref{4.16}, there is some $n$ such that $\E^{n+1}({^\bot\mathcal{D}}\cap\check{{_\mathcal{C}\X}},{_\mathcal{C}\X})=0$, and hence $\E(K_{m+n+1},K_{m+n})\cong\E^{n+1}(K_{m+n+1},K_{m})=0$.
This shows that the $\E$-triangle
$$\xymatrix@C=0.5cm{K_{m+n}\ar[r]&D_{m+n+1}\ar[r]&K_{m+n+1} \ar@{-->}[r]&}$$ is split, and hence $K_{m+n}\in\mathcal{D}$. Therefore,
$\check{{_\mathcal{C}\X}}\subseteq\check{\mathcal{D}}$.
\end{proof}

Collecting the above results, we now give the second type of Auslander-Reiten correspondence, which gives a classification for
$\mathcal{C}$-tilting subcategories in terms of a certain class of relative resolving strongly cotravariantly finite
subcategories in ${_\mathcal{C}\X}$.

\begin{theorem}\label{fal}
Let $\mathcal{C}$ be self-orthogonal such that ${_\mathcal{C}\X}$ has a relative injective cogenerator  and $\check{{_\mathcal{C}\X}}$ is finitely filtered.
Then there is a one-one correspondence as follows:
 \begin{align*}
   \left\{\mbox{finite  }\mathcal{C}\mbox{-tilting subcategories}\right\} &\longrightarrow {\left\{ \begin{array}{c}
                                                                                             \mbox{subcategories }\mathcal{D}\mbox{ which are relative} \\
                                                                                             \mbox{resolving, strongly contravariantly}\\ \mbox{finite in }{_\mathcal{C}\X},\mbox{ and are contained in }\hat{\mathcal{C}}
                                                                                           \end{array}\right\}.}\\
   \mathcal{T}&\mapsto \   \check{\mathcal{T}}\cap{_\mathcal{C}\X}
 \end{align*}
\end{theorem}

\begin{proof}
By Theorem \ref{4.12}, Propositions \ref{4.15} and \ref{4.17}, it suffices to show
$$
{^\bot({_\mathcal{T}\X})\cap{_\mathcal{C}\X}}=\check{\mathcal{T}}\cap{_\mathcal{C}\X}.
$$
Clearly, $\check{\mathcal{T}}\cap{_\mathcal{C}\X}\subseteq{^\bot({_\mathcal{T}\X})\cap{_\mathcal{C}\X}}$. On the other hand, by Proposition \ref{4.3} and Corollary \ref{3.4}, ${^\bot({_\mathcal{T}\X})\cap{_\mathcal{C}\X}}\subseteq{^\bot({_\mathcal{T}\X})\cap\check{{_\mathcal{C}\X}}}=
{^\bot({_\mathcal{T}\X})\cap{_\mathcal{T}\check{\X}}}$
which shows that $
{^\bot({_\mathcal{T}\X})\cap{_\mathcal{C}\X}}\subseteq\check{\mathcal{T}}\cap{_\mathcal{C}\X}.
$
\end{proof}

\begin{example}
\begin{itemize}
  \item[(1)] If $\C=\mbox{mod}A$, where $A$ is an Artin algebra and $\mbox{mod}A$ is the category of finitely generated left $A$-modules. We can get the Auslander-Reiten correspondence of $n$-tilting pairs in \cite[Theorems 3.9 and 3.15]{wx2} from Theorems \ref{4.12} and \ref{fal}.
  \item[(2)] Using Theorems \ref{4.12} and \ref{fal} to Example \ref{exampair}(5), we get the Auslander-Reiten correspondence of an $n$-tilting object in \cite{ZZhuang}.
\end{itemize}
\end{example}

\section*{Acknowledgment}

This work was supported by the NSF of China (Grant Nos. 11671221, 11971225, 11901341), the project ZR2019QA015 supported by Shandong Provincial Natural Science Foundation,  the project funded by China Postdoctoral Science
Foundation (No. 2020M682141),
and the Young Talents Invitation Program of Shandong Province.
Part of this work was done by the first author during a visit at Tsinghua University.
He would like to express his gratitude to  Department of Mathematical Sciences and especially
to Professor Bin Zhu for the warm hospitality and the excellent working conditions. He also thanks Panyue Zhou for his helpful suggestions.

\vspace{4mm}
{\footnotesize
\hspace{-1.6em}{\bf Tiwei Zhao}\\
School of Mathematical Sciences, Qufu Normal University,
273165 Qufu,  P. R. China.\\
E-mail: \verb"tiweizhao@qfnu.edu.cn"\\[0.1cm]
{\bf Bin Zhu}\\
Department of Mathematical Sciences, Tsinghua University,
100084 Beijing, P. R. China.\\
E-mail: \verb"zhu-b@mail.tsinghua.edu.cn"\\[0.1cm]
{\bf Xiao Zhuang}\\
Department of Mathematical Sciences, Tsinghua University,
100084 Beijing, P. R. China.\\
E-mail: \verb"zhuangx16@tsinghua.org.cn"}
\end{document}